\documentclass[aop,MSNbibl,seceqn,dvips]{arximspdf}

\doi{10.1214/12-AOP813} 
\volume{41}
\issue{5}
\pubyear{2013}
\firstpage{3420}
\lastpage{3461}

\makeatletter

\newcommand{\rrvert}{\vert}
\newcommand{\llvert}{\vert}
\newtheorem{theorem}{Theorem}[section]
\newtheorem{proposition}{Proposition}[section]
\newtheorem{corollary}{Corollary}[section]
\newtheorem{lemma}{Lemma}[section]
\newtheorem{problem}{Problem}[section]
\newproclaim{notation}{Notational convention}[section]
\newtheorem{claim}{Claim}[section]
\newproclaim{remark}{Remark}[section]

\newcommand{\thmref}[1]{Theorem~\ref{thm:#1}} 
\newcommand{\lemref}[1]{Lemma~\ref{lem:#1}} 
\newcommand{\propref}[1]{Proposition~\ref{prop:#1}} 
\newcommand{\claimref}[1]{Claim~\ref{claim:#1}} 

\newcommand{\secref}[1]{Section~\ref{sec:#1}} 
\newcommand{\eqnref}[1]{(\ref{eq:#1})} 

\newcommand\ignore[1]{}

\def\R{\mathbf{R}} 
\def\Z{\mathbf{Z}} 
\def\N{\mathbf{N}} 

\def\sF{\mathcal{F}}
\def\sG{\mathcal{G}}
\def\sO{\mathcal{O}}
\def\sP{\mathcal{P}}

\def\eps{\varepsilon}

\def\Expdist{\mathsf{Exp}}

\def\bE{\mathbf{V}}
\def\bbD{\mathbf{D}}

\def\Ctime{\mathsf{C}}
\def\bV{\mathbf{V}}
\makeatother

\begin{document}
\begin{frontmatter}

\title{Mean field conditions for coalescing random walks\thanksref{T1}}
\runtitle{Mean field coalescing r. w.}
\thankstext{T1}{Supported by a \textit{Universal} grant and a \textit
{Bolsa de Produtividade em Pesquisa} from CNPq, Brazil.}

\begin{aug}
\author{\fnms{Roberto Imbuzeiro} \snm{Oliveira}\corref{}\ead[label=e1]{rimfo@impa.br}}
\runauthor{R.~I. Oliveira}
\affiliation{IMPA}
\address{IMPA\\
Estrada Dona Castorina, 110\\
Rio de Janeiro, RJ\\
Brazil\\
\printead{e1}}
\end{aug}

\received{\smonth{9} \syear{2011}}
\revised{\smonth{6} \syear{2012}}

%
\begin{abstract}
The main results in this paper are about the \emph{full coalescence
time} $\Ctime$ of a system of coalescing random walks over a finite
graph $G$. Letting $\mathsf{m}(G)$ denote the mean meeting time of two such
walkers, we give sufficient conditions under which $\mathbf{E}
[\Ctime  ]\approx
2\mathsf{m}(G)$ and $\Ctime/\mathsf{m}(G)$ has approximately the
same law as in the
``mean field'' setting of a large complete graph. One of our theorems is
that mean field behavior occurs over all vertex-transitive graphs whose
mixing times are much smaller than $\mathsf{m}(G)$; this nearly solves
an open
problem of Aldous and Fill and also generalizes results of Cox for
discrete tori in $d\geq2$ dimensions. Other results apply to
nonreversible walks and also generalize previous theorems of Durrett
and Cooper et al. Slight extensions of these results apply to voter
model consensus times, which are related to coalescing random walks via duality.

Our main proof ideas are a strengthening of the usual approximation of
hitting times by exponential random variables, which give results for
nonstationary initial states; and a new general set of conditions
under which we can prove that the hitting time of a union of sets
behaves like a minimum of independent exponentials. In particular, this
will show that the first meeting time among $k$ random walkers has mean
$\approx\mathsf{m}(G)/ \bigl({k\atop2} \bigr)$.
\end{abstract}

%


\end{frontmatter}

\section{Introduction}\label{sec1}

Start a continuous-time random walk from each vertex of a finite,
connected graph $G$. The walkers evolve independently, except that when
two walkers \emph{meet}---that is, lie on the same vertex at the same
time---they coalesce into one. One may easily show that there will
almost surely be a finite time at which only one walk will remain in
this system. The first such time is called the \emph{full coalescence
time} for $G$ and is denoted by $\Ctime$.

The main goal of this paper is to show that one can estimate the law of
$\Ctime$ for a large family of graphs $G$, and that this law only
depends on $G$ through a single rescaling parameter. More precisely, we
will prove results of the following form: if the \emph{mixing time}
$t_{\mathrm{mix}}^G$ of $G$ (defined in \secref{prelim}) is ``small,''
then there
exists a parameter $\mathsf{m}(G)>0$ such that the law $\Ctime
/\mathsf{m}(G)$
takes a
universal shape. Slight extensions of these results will be used to
study the so-called \emph{voter model consensus time} on $G$.

The universal shape of $\Ctime/\mathsf{m}(G)$ comes from a \emph{mean field}
computation over a large complete graph $K_n$. In this case the
distribution of $\Ctime$ can be computed exactly (cf.~\cite{AldousFill_RWBook}, Chapter 14),
\[
\frac{\Ctime}{(n-1)/2}=_d\sum_{i=2}^n
\mathsf{Z}_{i},
\]
where:
%
%
\begin{equation}
\label{eq:defZi}\mbox{The $\mathsf{Z}_{i}$'s are
independent and }\forall i\geq2, t\geq0 \qquad \mathbf{P} (\mathsf{Z}_{i}
\geq t )=e^{-t({i\atop2})}.
\end{equation}
In words, $\Ctime$ is a rescaled sum of independent exponential random
variables with means $1/\bigl({i\atop2}\bigr)$, $2\leq i\leq n$.

The scaling factor $(n-1)/2$ is the expected meeting time of two
independent random walks over $K_n$,
and we see that
\[
\frac{\Ctime}{(n-1)/2}\to_w \sum_{i\geq2}
\mathsf{Z}_{i}\quad\mbox{and}\quad\frac{\mathbf{E} [\Ctime
]}{(n-1)/2}\to2\qquad\mbox{when $n$ grows.}
\]

This suggests the general problem we address in this paper:

%
\begin{problem}Given a graph $G$, let $\mathsf{m}(G)$ denote the expected
meeting time of two independent random walks over $G$, both started
from stationarity. Give sufficient conditons on $G$ under which $\Ctime
$ has mean-field behavior, that is,
%
%
\begin{equation}
\label{eq:meanfieldJJJ}\operatorname{Law} \bigl(\Ctime/\mathsf {m}(G) \bigr)
\approx\operatorname{Law} \biggl(\sum_{i\geq2}\mathsf
{Z}_{i} \biggr)
\end{equation}
and
%
%
\begin{equation}
\label{eq:aldousfill}\mathbf{E} [\Ctime ]\approx\mathsf {m}(G)\mathbf{E} \biggl[
\sum_{i\geq2}\mathsf{Z}_{i} \biggr] = 2
\mathsf{m}(G).
\end{equation}
\end{problem}

A version of this problem was posed in Aldous and Fill's 1994 draft
\cite{AldousFill_RWBook}, Chapter~14, and much more recently by Aldous
\cite{AldousTalk}. However, as far as we know there are only two
families of examples where the problem has been fully solved. Discrete
tori $G=(\Z/m\Z)^d$ with with $d\geq2$ fixed and $m\gg1$ were
considered in Cox's 1989 paper~\cite{Cox_Coalescing}. More recently,
Cooper, Frieze and Radzik~\cite{CooperEtAl_IPSOnExpanders} proved mean
field behavior in large random $d$-regular graphs ($d$~bounded).
Partial results were also obtained by Durrett \cite
{Durrett_RGDynamics,Durrett_PNAS} for certain models of large networks.

We note that mean-field behavior is not universal over all large
graphs. One counterexample comes from a sequence of growing cycles,
where the limiting law of $\Ctime$ was also computed by Cox \cite
{Cox_Coalescing}. Stars with $n$ vertices are also not mean field:
$\Ctime$ is lower bounded by the time the last edge of the star is
crossed by some walker, which is about $\log n$, whereas $\mathsf
{m}(G)$ is
uniformly bounded.

\subsection{Results for transitive, reversible chains}

Our results in this paper address~\eqnref{meanfieldJJJ} and \eqnref
{aldousfill} simultaneously by proving approximation bounds in $L_1$
Wasserstein distance, which implies closeness of first moments; cf.
\secref{prelim_W}.

The first theorem implies that mean field behavior occurs whenever $G$
is vertex-transitive, and its \emph{mixing time} (defined in \secref
{prelim}) is much smaller than $\mathsf{m}(G)$. This nearly solves a problem
posed by Aldous and Fill in~\cite{AldousFill_RWBook}, Chapter 14. In
their open Problem 12, they ask for an analogous result with the
relaxation time replacing the mixing time (more on this below).

The natural setting for this first theorem is that of walkers evolving
according to the same reversible, transitive Markov chain (the
definition of $\Ctime$ easily generalizes to this case), where \emph{transitive}
means that for any two states $x$ and $y$ one can find a
permutation of the state space mapping $x$ to $y$ and leaving the
transition rates invariant. Clearly, the standard continuous-time
random walk on a vertex-transitive graph is transitive in this sense.

%
\begin{notation}\label{not:bigoh}In this paper we will use ``$b=O
(a )$'' in the following sense: there exist universal constants
$C,\xi>0$
such that $|a|\leq\xi\Rightarrow|b|\leq C |a|$.
\end{notation}

%
\begin{theorem}[(Mean field for transitive, reversible chains)]\label{thm:transitive}
Let $Q$ be the (generator of a) transitive, reversible,
irreducible Markov chain over a finite state space $\bV$, with mixing
time $t_{\mathrm{mix}}^Q$. Define $\mathsf{m}(Q)$ to be the expected
meeting time of two
independent continuous-time random walks over $\bV$ that evolve
according to~$Q$, when both are started from stationarity. Denote by
$\Ctime$ the full coalescence time for walks evolving according to $Q$.
Finally, define $\{\mathsf{Z}_{i}\}_{i=2}^{+\infty}$ as in \eqnref
{defZi}. Then
\[
d_W \biggl(\operatorname{Law} \biggl(\frac{\Ctime}{\mathsf
{m}(Q)} \biggr),
\operatorname{Law} \biggl(\sum_{i\geq2}\mathsf
{Z}_{i} \biggr) \biggr)= O \biggl( \biggl[\rho(Q)\ln \biggl(
\frac
{1}{\rho(Q)} \biggr) \biggr]^{1/6} \biggr),
\]
where
\[
\rho(Q)\equiv\frac{t_{\mathrm{mix}}^Q}{\mathsf{m}(Q)},
\]
and $d_W$ denotes $L_1$ Wasserstein distance. In particular,
\[
\mathbf{E} [\Ctime ] = \biggl\{2 + O \biggl( \biggl[\rho (Q)\ln \biggl(
\frac {1}{\rho(Q) } \biggr) \biggr]^{1/6} \biggr) \biggr\} \mathsf{m}(Q).
\]
\end{theorem}

This result generalizes Cox's theorem~\cite{Cox_Coalescing} for $(\Z
/m\Z
)^d$ with $d\geq2$ and growing~$m$. In this case, for any fixed $d$,
the mixing time grows as $m^2$ whereas $\mathsf{m}(G)\approx m^2\ln m$ for
$d=2$ and $\mathsf{m}(G)\approx m^d$ for larger $d$. The original problem
posed by Aldous and Fill remains open, but we note that:
\begin{itemize}
\item For transitive, reversible chains, the mixing time is at most a
$C \ln|\bV|$ factor away from the relaxation time, with $C>0$
universal (this is true whenever the stationary distribution is
uniform). This means we are not too far off from a full solution;
\item Any counterexample to their problem would have to come from a
vertex-transitive graph with mixing time of the order of $\mathsf
{m}(G)$ \emph{and} relaxation time asymptotically smaller than the mixing time. To
the best of our knowledge, such an object is not known to exist.
\end{itemize}

\subsection{Results for other chains}

We also have results on coalescing random walks evolving according to
arbitrary generators $Q$ on finite state spaces $\bV$. Again, we only
require that the mixing time $t_{\mathrm{mix}}^Q$ of $Q$ be
sufficiently small
relative to other parameters of the chain.

%
\begin{theorem}[(Mean field for general Markov chains)]\label
{thm:general}Let $Q$ denote (the generator of) a mixing Markov chain
over a finite set $\bV$, with unique stationary distribution $\pi$.
Denote by $q_{\max}$ the maximum transition rate from any $x\in\bV$
and by $\pi_{\max}$ the maximum stationary probability of an element of~$\bV$.
Let $\mathsf{m}(Q)$ denote the expected meeting time of two
random walks
evolving according to $Q$, both started from $\pi$. Finally, let
$\Ctime
$ denote the full coalescence time of random walks evolving in $\bV$
according to $Q$. Then
\[
d_W \biggl(\operatorname{Law} \biggl(\frac{\Ctime}{\mathsf
{m}(Q)} \biggr),
\operatorname{Law} \biggl(\sum_{i\geq2}\mathsf
{Z}_{i} \biggr) \biggr)= O \biggl( \biggl(\alpha(Q) \ln \biggl(
\frac
{1}{\alpha(Q)} \biggr)\ln^4|\bV| \biggr)^{1/6} \biggr),
\]
where
\[
\alpha(Q) = \bigl(1+q_{\max} t_{\mathrm{mix}}^Q \bigr)
\pi_{\max},
\]
and $d_W$ again denotes $L^1$ Wasserstein distance. In particular,
\[
\mathbf{E} [\Ctime ] = \biggl\{2 + O \biggl( \biggl[\alpha (Q)\ln \biggl(
\frac {1}{\alpha (Q)} \biggr) \ln^4|\bV| \biggr]^{1/6} \biggr)
\biggr\} \mathsf{m}(Q).
\]
\end{theorem}
We note that this theorem does \emph{not} imply \thmref{transitive}: for
instance, it does \emph{not} work for two-dimensional discrete tori.
However, the well-known formula for $\pi$ over graphs gives the
following corollary:
%
%
\begin{corollary}[(Proof omitted)]\label{cor:network}Assume $G$ is a
connected graph with vertex set $\bV$, where each vertex $x\in\bV$ has
degree $\operatorname{ deg}_G(x)$. Assume that $\eps\in(|\bV|^{-1},1)$ is such that
\[
\biggl(\frac{\max_{x\in\bV}\operatorname{ deg}_G(x)}{|\bV|^{-1} \sum_{x\in
\bV
}\operatorname{ deg}_G(x)} \biggr) t_{\mathrm{mix}}^G\leq
\frac{\eps|\bV
|}{\ln^4|\bV
|\ln\ln
|\bV|}.
\]
Then
\[
d_W \biggl(\operatorname{Law} \biggl(\frac{\Ctime}{\mathsf
{m}(G)} \biggr),
\operatorname{Law} \biggl(\sum_{i\geq2}\mathsf
{Z}_{i} \biggr) \biggr) = O \biggl( \biggl[\eps \biggl(1 +
\frac{\ln
(1/\eps)}{\ln\ln|\bV |} \biggr) \biggr]^{1/6} \biggr).
\]
\end{corollary}
This corollary suffices to prove mean field behavior over a variety of
examples, such as:
\begin{itemize}
\item\emph{all graphs with bounded ratio of maximal to average degree
and mixing time at most of the order $|\bV|/\ln^5|\bV|$}: this includes
expanders~\cite{CooperEtAl_IPSOnExpanders} and supercritical
percolation clusters in $(\Z/m\Z)^d$ with $d\geq3$ fixed \cite
{BenjaminiMossel_RWPercolation,Pete_ConnectivityMixingPercolation};
\item\emph{all graphs with maximal degree $\leq|\bV|^{1-\eta}$
($\eta
>0$ fixed) and mixing time that is polylogarithmic in $|\bV|$}: this
includes the giant component of a typical Erd\"{o}s--R\'{e}nyi graph
$G_{n,d/n}$ with $d>1$~\cite{FountoulakisReed_RWGiantComponent} and the
models of large networks considered by Durrett \cite
{Durrett_RGDynamics,Durrett_PNAS}.
\end{itemize}

Let us briefly comment on the case of large networks. Durrett has
estimated $\mathsf{m}(G)$ in these models, and has proven results
similar to
ours for a bounded number of walkers. We do not attempt to compute
$\mathsf{m}(G)$ here, which in general is a model-specific parameter. However,
we do show that mean field behavior for $\Ctime$ follows from
``generic'' assumptions about networks that hold for many different
models. This is important because recent measurements of real-life
social networks~\cite{LeskovecEtAl_Community} suggest that known models
of large networks are very inaccurate with respect to most network
characteristics outside of degree distributions and conductance. In
fairness, coalescing random walks and voter models over large networks
are not particularly realistic either, but at the very least we know
that mean field behavior is not an artifact of a particular class of
models. We also observe that our \thmref{general} also works for
nonreversible chains, for example, random walks on directed graphs.

\subsection{Results for the voter model}\label{sec:votermodel}

The voter model is a very well-known process in the interacting
particle systems literature~\cite{Liggett_IPSBook}. The configuration
space for the voter model is the power set $\sO^{\bV}$ of functions
$\eta:\bV\to\sO$, where $\bV$ is some nonempty set, and $\sO$ is a
nonempty set of possible opinions. The evolution of the process is
determined by numbers $q(x,y)$ ($x,y\in V, x\neq y$) and is informally
described as follows: at rate $q(x,y)$, node $x$ copies $y$'s opinion.
That is, there is a transition at rate $q(x,y)$ from any state $\eta
\dvtx \bV
\to\sO$ to the corresponding state $\eta^{x\leftarrow y}$, where
\[
\eta^{x\leftarrow y}(z) = \cases{ %
\eta(y), & \quad $\mbox{if }z=x;$
\vspace*{2pt}\cr
\eta(z), &\quad $\mbox{for all other }z\in V\setminus\{x\}.$}
\]
A classical duality result relates this voter model to a system of
coalescing random walks with transition rates $q(\cdot,\cdot\cdot)$ and
corresponding generator $Q$. More precisely, suppose that $\bV=\{
x(1),\ldots,x(n)\}$ and that $(\overline{X}_t(i))_{t\geq0,1\leq i\leq
n}$ is a system of coalescing random walks evolving according to $Q$
with $X_0(i)=x(i)$
for each $1\leq i\leq n$.

%
\begin{proposition}[(Duality~\cite{AldousFill_RWBook})] Choose $\eta
_0\in
\sO^{\bV}$. Then the configuration
\[
\hat{\eta}_t\dvtx x(i)\in\bV\mapsto\eta_0 \bigl(
\overline{X}_t(i) \bigr)\in\sO\qquad (1\leq i\leq n)
\]
has the same distribution as the state $\eta_t$ of the voter model at
time $t$, when the initial state is $\eta_0$. In particular, the \emph{consensus time} for the voter model
\[
\tau\equiv\inf \bigl\{t\geq0 \dvtx \forall i,j\in\bV, \eta_t(i)=\eta
_t(j) \bigr\}
\]
satisfies $\mathbf{E} [\tau ]\leq\mathbf{E} [\Ctime
]<+\infty$.
\end{proposition}

Now assume that the initial state $\eta_0\in\sO^{\bV}$ is random and
that the random variables $\{\eta_0(x)\}_{x\in V}$ are i.i.d. and have
common law $\mu$ which is not a point mass. In this case one can show
via duality that the law of the consensus time $\tau$ is that of
$\Ctime
_{K\wedge n}$, where $K$ is a $\N$-valued random variable independent
of the coalescing random walks, defined by
\[
K = \min\{i\in\N\dvtx U_{i+1}\neq U_1\} \qquad\mbox{where
}U_1,U_2,U_3,\ldots,\mbox{ are i.i.d. draws
from $\mu$},
\]
and for each $1\leq k\leq n$,
\[
\Ctime_k\equiv\min \bigl\{t\geq0 \dvtx \bigl| \bigl\{\overline{X}_t(i)
\dvtx 1\leq i\leq n \bigr\}\bigr|=k \bigr\}.
\]
Thus the key step in analyzing the voter model via our techniques is to
prove approximations for the distribution of $\Ctime_k$. Theorems \ref
{thm:transitive} and~\ref{thm:general} imply mean-field behavior for
$\Ctime=\Ctime_1$. A quick inspection of the proofs reveals that the
same bounds for Wasserstein distance can be obtained for $\Ctime_k$ for
any $1\leq k\leq n$. It follows that:
%
%
\begin{theorem}[(Proof omitted)]Let $\bV,\sO$ and $\mu$ be as above, and
consider the voter model defined by $\bV$, $\sO$ and by the generator
$Q$ corresponding to transition rates $q(x,y)$.\vadjust{\goodbreak}
Assume that the sequence $\{\mathsf{Z}_{i}\}_{i\geq2}$ is defined as in
\eqnref{defZi}, and also that~$K$ has the law described above and is
independent from the $\mathsf{Z}_{i}$. Define $\rho(Q)$ and $\alpha
(Q)$ as in
Theorems~\ref{thm:transitive} and~\ref{thm:general}. Then the consensus
time $\tau$ for this voter model satisfies
\[
d_W \biggl(\operatorname{Law} \biggl(\frac{\tau}{\mathsf
{m}(Q)} \biggr),
\operatorname{Law} \biggl(\sum_{i>K}Z_i
\biggr) \biggr)= O \bigl( \bigl(\rho(Q)\ln\bigl(1/ \rho(Q)\bigr)
\bigr)^{1/6} \bigr)
\]
if $Q$ is reversible and transitive, and
\[
d_W \biggl(\operatorname{Law} \biggl(\frac{\tau}{\mathsf
{m}(Q)} \biggr),
\operatorname{Law} \biggl(\sum_{i>K}Z_i
\biggr) \biggr)=O \bigl( \bigl(\alpha(Q) \ln \bigl(1/\alpha(Q) \bigr) \ln
^4|\bV| \bigr)^{1/6} \bigr),
\]
otherwise.
\end{theorem}

%
%
%

\subsection{Main proof ideas}

Our proofs of Theorems~\ref{thm:transitive} and~\ref{thm:general} both
start from the formula \eqnref{defZi} for the terms in the distribution
of $\Ctime$ over $K_n$. Crucially, each term~$\mathsf{Z}_{i}$ has a specific
meaning: $\mathsf{Z}_{i}$ is the time it takes for a system with $i$ particles
to evolve to a system with $i-1$ particles, rescaled by the expected
meeting time of two walkers. For $i=2$, this is just the (rescaled)
meeting time of a pair of particles, which is an exponential random
variable with mean $1$. For $i>2$, we are looking at the first meeting
time among $\bigl({i\atop2}\bigr)$ pairs of particles. It turns out that these
pairwise meeting times are independent; since the minimum of $k$
independent exponential random variables with mean $\mu$ is an
exponential r.v. with mean $\mu/k$, we deduce that $\mathsf{Z}_{i}$ is
exponential with mean $1/\bigl({i\atop2}\bigr)$.

The bulk of our proof consists of proving something similar for more
general chains $Q$. Fix some such $Q$, with state space $\bV$, and let
$\Ctime_i$ denote the time it takes for a system of coalescing random
walks evolving according to $Q$ to have~$i$ uncoalesced particles.
Clearly, $M\equiv\Ctime_1-\Ctime_2$ is the meeting time of a pair of
particles, which is the hitting time of the diagonal set
\[
\Delta\equiv \bigl\{(x,x) \dvtx x\in\bV \bigr\}
\]
by the Markov chain $Q^{(2)}$ given by a pair of independent
realizations of $Q$. More generally, $M^{(i+1)}=\Ctime_{i}-\Ctime
_{i+1}$ is the hitting time of
\[
\Delta^{(i+1)}= \bigl\{ \bigl(x(1),\ldots,x(i+1) \bigr) \dvtx \exists 1\leq
i_1<i_2\leq i+1, x(i_1)=x(i_2)
\bigr\}.
\]
The mean-field picture suggests that each $M^{(i+1)}$ should be close
in distribution to~$\mathsf{Z}_{i}$. Indeed, it is known that:

\textit{General principle}: Let $H_A$ be the hitting time of a subset $A$
of states. If the mixing time $t_{\mathrm{mix}}^Q$ is small relative
to $\mathbf{E} [H_A ]$, then $H_A$ is approximately
exponentially distributed.

This is a general meta-result for small subsets of the state space of a
Markov chain; precise versions (with different quantitative bounds) are
proven in~\cite{AldousBrown_RareEvents,Aldous_ExpHittingTimes} when the
chain starts from the stationary distribution. However, we face a few
difficulties when trying to use these off-the-shelf results:
\begin{longlist}[(1)]
\item[(1)] For each $i$, $M^{(i+1)}$ is the first hitting time of
$\Delta^{(i+1)}$ after time $\Ctime_{i+1}$. The random walkers are
\emph{not} stationary at this random time, so we need to ``do'' exponential
approximation from nonstationary starting points.
\item[(2)] In order to get Wasserstein approximations, we need better
control of the tail of $M^{(i+1)}$.
\item[(3)] To prove that $\mathsf{Z}_{i}$ and $M^{(i+1)}/\mathsf
{m}(Q)$ are close, we
must show something like that $\mathbf{E} [M^{(i+1)}
]\approx\mathbf{E} [M ]/
\bigl({i+1\atop2}\bigr)$, that is, that $M^{(i+1)}$ behaves like the minimum of
$\bigl({i+1 \atop2}\bigr)$ independent exponentials.
\item[(4)] Finally, we should not expect the exponential approximation
to hold when $\Delta^{(i+1)}$ is too large. That means that the ``big
bang'' phase (to use Durrett's phrase) at the beginning of the process
has to be controlled by other means.
\end{longlist}

It turns out that we can deal with points 1 and 2 via a different kind
exponential approximation result, stated as \thmref{rareeventsI}. This
result will give bounds of the following form:
%
%
\begin{equation}
\label{eq:definformalExpdist}\mathbf{P}_{x} (H_A>t ) =
\bigl(1+o (1 )\bigr)\exp \biggl(-\frac{t}{(1+o (1 ))\mathbf{E} [H_A ]} \biggr)
\end{equation}
as long as
\[
t_{\mathrm{mix}}^Q=o \bigl(\mathbf{E} [H_A ] \bigr)
\quad\mbox {and}\quad\mathbf{P}_{x} \bigl(H_A\leq
t_{\mathrm{mix}}^Q \bigr)=o (1 ).
\]
Notice that this holds even for \emph{nonstationary} starting points $x$
if the chain started from $x$ is unlikely to hit $A$ before the mixing
time. This is discussed in \secref{nearlyExp} below. We also take some
time in that section to develop a specific notion of ``near exponential
random variable.'' Although this takes up some space, we believe it
provides a useful framework for tackling other problems. We note that a
version of \thmref{rareeventsI} for stationary initial states result is
implicit in~\cite{Aldous_ExpHittingTimes}.

We now turn to point $3$. The key difficulty in our setting is that,
unlike Cox~\cite{Cox_Coalescing} or Cooper et al. \cite
{CooperEtAl_IPSOnExpanders}, we do not have a good ``local'' description
of the graphs under consideration which we could use to compute
$\mathbf{E} [M^{(i+1)} ]$ directly. We use instead a simple
general idea, which we
believe to be new, to address this point. Clearly, $M^{(i+1)}$ is a
minimum of $\bigl({i+1 \atop2}\bigr)$ hitting times. Let us consider the \emph{general} problem of understanding the law of
\[
H_B = \min_{1\leq i\leq\ell}H_{B_i}\qquad\mbox{where }B= \bigcup_{i=1}^\ell B_i,
\]
under the assumption that $\mathbf{E} [H_{B_i} ]=\mu$ does
not depend on $i$
when the initial distribution is stationary (this covers the case of
$M^{(i+1)}$). Assume also that \eqnref{definformalExpdist} holds for
all $A\in\{B,B_1,B_2,\ldots,B_\ell\}$. Then the following holds for
$\eps$ in a suitable range:
\[
\forall A\in\{B,B_1,B_2,\ldots,B_\ell\} \qquad \mathbf{P} \bigl(H_A\leq\eps\mathbf{E} [H_A ] \bigr)
\approx\eps.
\]
Morally speaking, this means that $\eps\mathbf{E} [H_A ]$
is the $\eps$-quantile
of $H_A$ for all $A$ as above; this is implicit in \cite
{Aldous_ExpHittingTimes} and is made explicit in our own \thmref
{rareeventsI}. Now apply this to $A=B$, with $\eps$ replaced by $\eps
\mu
/\mathbf{E} [H_B ]$, and obtain
\[
\frac{\eps\mu}{\mathbf{E} [H_B ]}\approx\mathbf {P} (H_B\leq\eps\mu ) = \mathbf{P}
\Biggl(\bigcup_{i=1}^\ell
\{H_{B_i}\leq\eps\mu \} \Biggr).
\]
If we can show that the pairwise correlations between the events $\{
H_{B_i}\leq\eps\mu\}$ are sufficiently small, then we may obtain
\[
\frac{\eps\mu}{\mathbf{E} [H_B ]}\approx\mathbf{P} \Biggl(\bigcup
_{i=1}^\ell \{ H_{B_i}\leq\eps\mu\} \Biggr)
\approx\sum_{i=1}^\ell \mathbf{P}
(H_{B_i}\leq\eps\mu ) = \ell\eps.
\]
This gives
\[
\mathbf{E} [H_B ]\approx\frac{\mu}{\ell}
\]
as if the times $H_{B_1},\ldots,H_{B_\ell}$ were independent
exponentials. The reasoning presented here is made rigorous and
quantitative in \thmref{rareeventsII} below.

Finally, we need to take care of point 4, that is, the ``big bang''
phase. In the setting of \thmref{general}, we simply use our results on
the coalescence times for smaller number of particles, which seems
wasteful but is enough to prove our results. For the
reversible/transitive case, we use a bound from~\cite{Oliveira_TAMS}
which is of the optimal order. Incidentally, the differences in the
bounds of the two theorems come from this better bound for the big bang
phase and from a more precise control of the correlations between
meeting times of different pairs of walkers.

\subsection{Outline}

The remainder of the paper is organized as follows. \secref{prelim}
contains several preliminaries. \secref{nearlyExp} contains a general
discussion of random variables with nearly exponential distribution and
our general approximation results for hitting times. In \secref
{meeting} we apply these results to the first meeting time among $k$
particles, after proving some technical estimates. \secref{CRW}
contains the formal definition of the coalescing random walks process
and proves mean field behavior for a moderate initial number of
walkers. Finally, \secref{proofsmain} contains the proofs of Theorems
\ref{thm:transitive} and~\ref{thm:general}. Related results and open
problems are discussed in the final sections.

\section{Preliminaries}\label{sec:prelim}

\subsection{Basic notation}\label{sec:notation}

We write $\N$ for nonnegative integers and $[k]=\{1,2,\ldots, k\}$ for
any $k\in\N\setminus\{0\}$. Given a set $S$, we let $|S|$ denote its
cardinality. Moreover, for $k\in\N$, we let
\[
\pmatrix{S\cr k}\equiv\bigl\{A\subset S \dvtx |A|=k\bigr\}.
\]
Notice that with this notation,
\[
|S|\mbox{ finite}\Rightarrow \biggl\llvert\pmatrix{S\cr k} \biggr\rrvert=
\pmatrix{|S|\cr
k} \equiv\frac{|S|!}{k! (|S|-k)!}.
\]

We will often speak of universal constants $C>0$. These are numbers
that do not depend on any of the parameters or mathematical objects
under consideration in a given problem. We will also use the notation
``$a=O (b )$'' in the universal sense prescribed in Notational
convention~\ref{not:bigoh}. In this way we can write down expressions
such as
\[
e^{b} = 1+ b + O \bigl(b^2 \bigr)\quad\mbox{and}\quad\ln \biggl(
\frac{1}{1-b} \biggr) = b+O \bigl(b^2 \bigr) = O (b ).
\]

Given a finite set $S$, we let $M_1(S)$ denote the set of all
probability measures over~$S$. Given $p,q\in M_1(S)$, their total
variation disance is defined as follows:
\[
d_\mathrm{ TV}(p,q)\equiv\frac{1}{2}\sum_{s\in S}\bigl|p(s)-q(s)\bigr|
= \sup_{A\subset S} \bigl[p(A) - q(A) \bigr],
\]
where $p(A) = \sum_{a\in A}p(a)$. For $S$ not finite, $M_1(S)$ will
denote the set of all probability measures over the ``natural'' $\sigma
$-field over $S$. For instance, for $S=\R$ we consider the Borel
$\sigma
$-field, and for $S=\bbD([0,+\infty),\bE)$ (see \secref{trajectories}
for a definition) we use the $\sigma$-field generated by projections.

If $X$ is a random variable taking values over $S$, we let
$\operatorname{Law} (X )\in
M_1(S)$ denote the distribution (or law) of $X$. Here we again assume
that there is a ``natural'' $\sigma$-field to work with.

\subsection{Wasserstein distance}\label{sec:prelim_W}

The \textit{$L_1$ Wasserstein distance} is a metric over probability
measures over $\R$ with finite first moments, given by
\[
d_W(\lambda_1,\lambda_2)=\int
_{\R} \bigl|\lambda_1(x,+\infty] - \lambda
_2(x,+\infty]\bigr| \,dx \qquad\bigl(\lambda_1,\lambda_2
\in M_1(\R) \bigr).
\]
A classical duality result gives
\[
d_W(\lambda_1,\lambda_2)= \sup
_{f\dvtx \R\to\R\mathrm{
1\mbox{-}Lipschitz}} \biggl(\int_{\R}f(x)
\lambda_1(dx)-\int_{\R}f(x)\lambda
_2(dx) \biggr).
\]

%
\begin{notation}Whenever we compute Wasserstein distances,
we will \emph{assume} that the distributions involved have first moments. This can be
checked in each particular case.
\end{notation}

%
\begin{remark}\label{rem:wasserstein}If $Z_1,Z_2$ are random
variables, we sometimes write
\[
d_W(Z_1,Z_2)\mbox{ instead of
}d_W\bigl(\operatorname{Law} (Z_1 ),\operatorname{Law}
(Z_2 )\bigr).
\]
Note that
\[
d_W(Z_1,Z_2) =\int_{\R}\bigl|
\mathbf{P} (Z_1\geq t ) - \mathbf{P} (Z_2\geq t )\bigr| \,dt.
\]
Also notice that
\[
\bigl|\mathbf{E} [Z_1 ]-\mathbf{E} [Z_2 ]\bigr|\leq
d_W(Z_1,Z_2).
\]
This is an equality if $Z_1\geq0$ a.s. and $Z_2=C Z_1$ for some
constant $C>0$,
%
%
\begin{equation}
\label{eq:wassersteinmultconstant}\forall C\in\R\qquad d_W(Z_1,C
Z_1) = |C-1|\mathbf{E} [Z_1 ],
\end{equation}
since $|f(C Z_1) - f(Z_1)|\leq|C-1| Z_1$ for every $1$-Lipschitz
function $f\dvtx \R\to\R$.
\end{remark}

We note here three useful lemmas on Wasserstein distance. These are
probably standard, but we could not find references for them, so we
provide proofs for the latter two lemmas in \secref{proof_W} of the
\hyperref[sec:proof_W]{Appendix}. The first lemma is immediate.

%
%
\begin{lemma}[(Sum lemma for Wasserstein distance; Proof omitted)]\label
{lem:sumW}For any two random variables $X,Y$ with finite first moments
and defined on the same probability space,
\[
d_W(X,X+Y)\leq\mathbf{E}\bigl [|Y| \bigr].
\]
\end{lemma}

For the next lemma, recall that, given two real-valued random variables
$X,Y$, we say that $X$ is stochastically dominated by $Y$ and write
$X\preceq_d Y$ if $\mathbf{P} (X>t )\leq\mathbf{P}
(Y>t )$ for all $t\in\R$.

%
\begin{lemma}[(Sandwich lemma for Wasserstein distance)]\label
{lem:sandwich}Let $Z$, $Z_-$, $Z_+$ and $W$ be real-valued random
variables with finite first moments and $Z_-\preceq_d Z\preceq_dZ_+$. Then
\[
d_W(Z,W)\leq d_W(Z_-,W) + d_W(Z_+,W).
\]
\end{lemma}

%
\begin{lemma}[(Conditional lemma for Wasserstein distance)]\label
{lem:conditionalwasserstein}Let $W_1$, $W_2$, $Z_1$, $Z_2$ be
real-valued random variables with finite first moments. Assume that
$Z_1$ and $Z_2$ independent and that $W_1$ is $\sG$-measurable for some
sub-$\sigma$-field $\sG$. Then
\begin{eqnarray*}
&& d_W\bigl(\operatorname{Law} (W_1+W_2 ),
\operatorname{Law} (Z_1+Z_2 )\bigr)
\\
&&\quad \leq d_W\bigl(\operatorname{Law} (W_1 ),
\operatorname{Law} (Z_1 )\bigr) + \mathbf{E} \bigl[d_W
\bigl(\operatorname{Law} (W_2\mid\sG ),\operatorname{Law}
(Z_2 )\bigr) \bigr].
\end{eqnarray*}
\end{lemma}
%
%
\begin{remark} Here we are implicitly assuming that $\operatorname
{Law} (W_2\mid \sG  )$ is given by some regular
conditional probability distribution.
\end{remark}

\subsection{Continuous-time Markov chains}

\subsubsection{State space and trajectories}\label{sec:trajectories}
Let $\bV$ be some nonempty finite set, called the \emph{state space}.
We write $\bbD\equiv\bbD([0,+\infty),\bV)$ for the set of all paths
\[
\omega\dvtx t\geq0\mapsto\omega_t\in\bV
\]
for which there exist $0=t_0<t_1<t_2<\cdots<t_n<\cdots$ with
$t_n\nearrow
+\infty$ and $\omega$ constant over each interval $[t_n,t_{n+1})$
$(n\in
\N)$. Such paths will sometimes be called \emph{c\`{a}dl\`{a}g}.

For each $t\geq0$, we let $X_t\dvtx \bbD\to\bV$ be the projection map
sending $\omega$ to $\omega_t$. We also define $X=(X_t)_{t\geq0}$ as
the identity map over $\bbD$. Whenever we speak about probability
measures and events over $\bbD$, we will implicitly use the $\sigma
$-field $\sigma(\bbD)$ generated by the maps $X_t$, $t\geq0$. We
define an associated filtration as follows:
\[
\sF_t\equiv\sigma\{X_s \dvtx 0\leq s\leq t\}\qquad (t\geq0).
\]

We also define the time-shift operators
\[
\Theta_T\dvtx \omega(\cdot)\in\bbD\mapsto\omega(\cdot+T)\in\bbD\qquad(T
\geq0).
\]

\subsubsection{Markov chains and their generators}
Let $q(x,y)$ be nonnegative real numbers for each pair $(x,y)\in\bV
^2$ with $x\neq y$. Define a linear operator $Q\dvtx \R^{\bV}\to\R^{\bV}$,
which maps $f\in\R^{\bV}$ to $Qf\in\R^{\bV}$ satisfying
\[
(Qf) (x)\equiv\sum_{y\in\bV\setminus\{x\}}q(x,y) \bigl(f(x)-f(y)
\bigr)\qquad (x\in\bV).
\]

It is a well-known result that there exists a unique family of
probability measures $\{\mathbf{P}_{x}\}_{x\in\bV}$ with the properties
listed below:
\begin{longlist}[(1)]
\item[(1)] for all $x\in\bV$, $\mathbf{P}_{x} (X_0=x )=1$;
\item[(2)] for all distinct $x,y\in\bV$, $\lim_{\eps\searrow
0}\frac
{\mathbf{P}_{x} (X_\eps=y )}{\eps} = q(x,y)$;
\item[(3)] \textit{Markov property}: for any $x\in\bV$ and $T\geq
0$, the
conditional law of $X\circ\Theta_T$ given $\sF_T$ under measure
$\mathbf{P}_{x}$ is given by $\mathbf{P}_{X_T}$.
\end{longlist}
The family $\{\mathbf{P}_{x}\}_{x\in\bV}$ satisfying these
properties is the
\emph{Markov chain with generator~$Q$}. We will often abuse notation and
omit any distinction between a Markov chain and its generator in our notation.

For $\lambda\in M_1(\bV)$, $\mathbf{P}_{\lambda}$ denotes the mixture
\[
\mathbf{P}_{\lambda}\equiv\sum_{x\in\bV}\lambda(x)
\mathbf{P}_{x}.
\]
This corresponds to starting the process from a random state
distributed according to $\lambda$.
For $x\in\bV$ or $\lambda\in M_1(\bV)$ and $Y\dvtx \bbD\to S$ a random
variable, we let $\operatorname{Law}_{x} (Y )$ or
$\operatorname{Law}_{\lambda} (Y )$ denote the law of
$Y$ under $\mathbf{P}_{x}$ or $\mathbf{P}_{\lambda}$ (resp.).

\subsubsection{Stationary measures and mixing}

Any Markov chain $Q$ as above has at least one stationary measure $\pi
\in M_1(\bV)$; this is a measure such that for any $T\geq0$,
\[
\operatorname{Law}_{\pi} (X\circ\Theta_T ) =
\operatorname{Law}_{\pi} (X ).
\]
We will be only interested in \emph{mixing Markov chains}, which are
those $Q$ with a unique stationary measure that satisfy the following condition:
\[
\forall\alpha\in(0,1), \exists T\geq0,\forall x\in\bV\qquad  d_\mathrm{
TV}\bigl(
\operatorname{Law} (x ) {X_T},\pi\bigr)\leq\alpha.
\]
The smallest such $T$ is called the $\alpha$-mixing time of $Q$ and is
denoted by $t_{\mathrm{mix}}^Q(\alpha)$. By the Markov property and
the definition
of total-variation distance, we also have that for all $\alpha\in
(0,1)$, all $t\geq t_{\mathrm{mix}}^Q(\alpha)$, all $x\in\bV$ and
all events $S$,
\[
\bigl|\mathbf{P}_{x} (X\circ\Theta_t\in S ) -
\mathbf{P}_{\pi
} (X\in S )\bigr|\leq\alpha.
\]
The specific value $t_{\mathrm{mix}}^Q\equiv t_{\mathrm{mix}}^Q(1/4)$
is called \emph{the}
mixing time of $Q$. We note that for all $\eps\in(0,1/2)$,
%
%
\begin{equation}
\label{eq:mixingfromstandard}t_{\mathrm{mix}}^Q(\eps)\leq C \ln (1/
\eps) t_{\mathrm{mix}}^Q,
\end{equation}
where $C>0$ is universal; this is proven in~\cite{LevinPeresWilmer_MCBook}, Section
4.5, for discrete time chains, but the same
argument works here.

\subsubsection{Product chains}\label{sec:productchains}

Letting $Q$ be as above, we may consider the joint trajectory of $k$
independent realizations of $Q$,
\[
X^{(k)}_t = \bigl(X_t(1),
\ldots,X_t(k) \bigr) \qquad (t\geq0)
\]
where each $(X_t(i))_{t\geq0}$ has law $\mathbf{P}_{x(i)}$. It turns out
that this corresponds to a Markov chain $Q^{(k)}$ on $\bV^k$ with
transition probabilities
\begin{eqnarray*}
&& q^{(k)} \bigl(x^{(k)},y^{(k)} \bigr) \\
&&\qquad= \cases{
q \bigl(x(i),y(i) \bigr), &\quad $\mbox{if }x(i)\neq y(i)
\wedge\forall j\in[k]\setminus\{i\}, x(j)= y(j);$
\vspace*{2pt}\cr
0, & \quad$\mbox{otherwise}.$}
\end{eqnarray*}

%
\begin{remark} In what follows we will always denote elements of
$\bV
^k$ [resp., $M_1(\bV^{k})$] by symbols like $x^{(k)},y^{(k)}, \ldots$
(resp., $\lambda^{(k)},\rho^{(k)},\ldots$). We will then denote the
distribution of $Q^{(k)}$ started from $x^{(k)}$ or $\lambda^{(k)}$ by
$\mathbf{P}_{x^{(k)}}$ or $\mathbf{P}_{\lambda^{(k)}}$. This is a
slight abuse of
our convention for the $Q$ chain, but the initial state/distribution
will always make it clear that we are referring to the product chain.
\end{remark}

The following result on $Q^{(k)}$ will often be useful.
%
%
\begin{lemma}\label{lem:productchains}Assume $Q$ is mixing and has
(unique) stationary distribution $\pi$. Then $Q^{(k)}$ is also mixing,
and the product measure $\pi^{\otimes k}$ is its (unique) stationary
distribution. Moreover, the mixing times of $Q^{(k)}$ satisfy
\[
\forall\alpha\in(0,1/2)\qquad t_{\mathrm{mix}}^{Q^{(k)}}(\alpha)\leq
t_{\mathrm{mix}}^Q( \alpha/k)\leq C \ln(k/\alpha)
t_{\mathrm{mix}}^Q
\]
with $C>0$ universal.
\end{lemma}
\begin{pf*}{Proof sketch}Notice that the law of $X^{(k)}_T$ has a
product form
\[
\operatorname{Law}_{x^{(k)}} \bigl(X^{(k)}_T \bigr)
= \operatorname {Law}_{x(1)} (X_T )\otimes
\operatorname{Law}_{x(2)} (X_T )\otimes\cdots\otimes
\operatorname{Law}_{x(k)} (X_T ).
\]
It is well known (and not hard to show) that the total-variation
distance between product measures is at most the sum of the distances
of the factors. This gives
\[
d_\mathrm{ TV} \bigl(\operatorname{Law}_{x^{(k)}} \bigl(X^{(k)}_T
\bigr), \pi^{\otimes k} \bigr)\leq\sum_{i=1}^kd_\mathrm{ TV}
\bigl(\operatorname{Law}_{x(i)} (X_T ),\pi \bigr).
\]
The RHS is $\leq\alpha$ if each term in the sum is less than $\alpha
/k$. This is achieved when $T\geq t_{\mathrm{mix}}^Q(\alpha/k)$;
\eqnref
{mixingfromstandard} then finishes the proof.\
\end{pf*}

\section{Nearly exponential hitting times}\label{sec:nearlyExp}

\subsection{Basic definitions}
We first recall a standard definition: the \emph{exponential
distribution} with mean $m>0$, denoted by $\Expdist(m)$, is the unique
probabilty dstribution $\mu\in M_1(\R)$ such that, if $Z$ is a random
variable with law $\mu$,
\[
\mathbf{P} (Z\geq t ) = e^{-t/m} \qquad(t\geq0).
\]
We write $Z=_d\Expdist(m)$ when $Z$ is a random variable with
$\operatorname{Law} (Z )=\Expdist(m)$.

Similarly, given $m>0$ as above and parameters $\alpha>0,\beta\in
(0,1)$, we say that a measure $\mu\in M_1(\R)$ has distribution
$\Expdist(m,\alpha,\beta)$ if it is the law of a random variable
$\widetilde{Z}$ with $\widetilde Z\geq0$ almost surely, and for all $t>0$,
\[
(1-\alpha) e^{-{t}/{((1-\beta)m)}}\leq\mathbf{P} (\widetilde Z\geq t )\leq(1+
\alpha)e^{-{t}/{((1+\beta)m)}}.
\]
We will write $\mu=\Expdist(m,\alpha,\beta)$ or $\widetilde{Z}=_d
\Expdist(m,\alpha,\beta)$ as a shorthand for this. Notice that
$\Expdist
(m,\alpha,\beta)$ does not denote a single distribution, but rather a
family of distributions that obey the above property, but we will
mostly neglect this minor issue.

Random variables with law $\Expdist(m,\alpha,\beta)$ will naturally
appear in our study of hitting times of Markov chains. We compile here
some simple results about them. The first proposition is trivial and we
omit its proof.

%
\begin{proposition}[(Proof omitted)]\label{prop:changemean}If $\mu\in
M_1(\R)$ satisfies
\[
\mu=\Expdist(m,\alpha,\beta),
\]
and $m'>0,\gamma\in(0,1)$ are such that $\beta+\gamma+\beta\gamma<1$,
\[
(1-\gamma) m'\leq m\leq(1+\gamma) m',
\]
then
\[
\mu=\Expdist \bigl(m',\alpha,\beta+\gamma+\beta\gamma \bigr).
\]
\end{proposition}

We now show that random variables $\Expdist(m,\alpha,\beta)$ are close
to the corresponding exponentials.

%
\begin{lemma}[{[Wasserstein distance error for $\Expdist(m,\alpha,\beta
)$]}]\label{lem:approxwasserstein}We have the following inequality for
all $\alpha>0$, $0<\beta<1$:
\[
d_W \bigl(\Expdist(m),\Expdist(m,\alpha,\beta) \bigr)\leq2(\alpha+
\beta) m.
\]
That is, if $\widetilde{Z}=_d \Expdist(m,\alpha,\beta)$, the
Wasserstein distance between $\operatorname{Law} (\widetilde
{Z} )$ and $\Expdist(m)$ is
at most $2\alpha m+2\beta m$.
\end{lemma}
\begin{pf}Assume $\widetilde{Z}=_d\Expdist(m,\alpha,\beta)$ and
$Z=_d\Expdist(m)$ are given. By convexity,
\begin{eqnarray*}
d_W(\widetilde{Z},Z) &=& \int_0^{+\infty}\bigl|
\mathbf{P} (\widetilde{Z}\geq t )-e^{-{t}/{m}}\bigr| \,dt
\\
&\leq& \int_{0}^{\infty
}\max_{\xi\in\{-1,+1\}}\bigl|(1+
\xi\alpha)_+e^{-{t}/{((1+\xi\beta
)m)}}-e^{-{t}/{m}}\bigr| \,dt
\\
&\leq& \int_{0}^{\infty}\bigl|(1+\alpha)e^{-{t}/{((1+\beta
)m)}}-e^{-{t}/{m}}\bigr|
\,dt
\\
& &{} + \int_{0}^{\infty}\bigl|(1-\alpha)_+
e^{-
{t}/{((1-\beta)m)}}-e^{-{t}/{m}}\bigr| \,dt
\\
&=:& (I) + (\mathit{II}).
\end{eqnarray*}
For the first term on the RHS, we note that
\[
\forall t\geq0\qquad (1+\alpha)e^{-{t}/{((1+\beta)m)}}-e^{-{t}/{m}}\geq0,
\]
hence
\[
(I) = \int_{0}^{\infty} \bigl\{(1+
\alpha)e^{-{t}/{((1+\beta
)m)}}-e^{-{t}/{m}} \bigr\} \,dt = [\alpha+\beta+\alpha\beta]
m.
\]
Similarly, for term $(\mathit{II})$ we have
\[
\forall t\geq0\qquad (1-\alpha)_+ e^{-{t}/{((1-\beta)m)}}-e^{-{t}/{m}}\leq0
\]
hence
\[
(\mathit{II}) = \int_{0}^{\infty} \bigl\{e^{-{t}/{m}} -
(1-\alpha)_+e^{-
{t}/{((1-\beta)m)}} \bigr\} \,dt \leq[\alpha+\beta-\alpha\beta]m.
\]
Hence
\[
d_W(\widetilde{Z},Z)\leq(I) + (\mathit{II}) = 2(\alpha+\beta) m.
\]
\upqed\end{pf}

\ignore

\subsection{Hitting times are nearly exponential}

In this section we consider a mixing continuous-time Markov chain $\{
\mathbf{P}_{x}\}_{x\in\bE}$ with generator $Q$, taking values over a finite
state space $\bE$, with unique stationary distribution $\pi$. Given a
nonempty $A\subset\bE$ with $\pi(A)>0$, we define the \emph{hitting
time of $A$} to be
\[
H_A(\omega)\equiv\inf \bigl\{t\geq0 \dvtx \omega(t)\in A \bigr\}\qquad \bigl(
\omega\in\bbD \bigl([0,+\infty ),\bE \bigr)\bigr).
\]
The condition $\pi(A)>0$ ensures that $\mathbf{E}_{x} [H_A
]<+\infty$ for all
$x\in\bE$.

Our first result in this section presents sufficient conditions on $A$
and $\mu\in M_1(\bE)$ that ensure that $H_A$ is approximately
exponentially distributed.

%
\begin{theorem}\label{thm:rareeventsI}In the above Markov chain
setting, assume that $0<\eps<\delta<1/5$ are such that
\[
\mathbf{P}_{\pi} \bigl(H_A\leq t_{\mathrm{mix}}^Q(
\delta\eps ) \bigr)\leq\delta\eps.
\]
Let $t_\eps(A)$ be the $\eps$-quantile of $\operatorname{Law}_{\pi
} (H_A )$, that is,
the unique number $t_\eps(A)\in[0,+\infty)$ with $\mathbf{P}_{\pi
} (H_A\leq t_\eps(A) )=\eps$ [this is well defined since
$\mathbf{P}_{\pi} (H_A\leq t )$ is
a continuous and strictly increasing function of $t$ in our
setting]. Given $\lambda\in M_1(\bE)$, write
\[
r_\lambda\equiv\mathbf{P}_{\lambda} \bigl(H_A\leq
t_{\mathrm
{mix}}^Q( \delta\eps) \bigr).
\]
Then
\[
\operatorname{Law}_{\lambda} (H_A )=\Expdist \biggl(
\frac
{t_\eps(A)}{\eps},O (\eps )+2r_\lambda, O (\delta ) \biggr).
\]
Moreover,
\[
\biggl\llvert\frac{\eps\mathbf{E}_{\pi} [H_A ]}{t_\eps
(A)}-1 \biggr\rrvert=O (\delta )
\]
and
\[
\operatorname{Law}_{\lambda} (H_A )=_d \Expdist
\bigl(\mathbf {E}_{\pi} [H_A ],O (\eps )+2r_\lambda,
O (\delta ) \bigr).
\]
\end{theorem}

We emphasize that results similar to this are not new in the literature
\cite{Aldous_ExpHittingTimes,AldousBrown_RareEvents}, but the
lower-tail part of our result does not\vadjust{\goodbreak} seem to be explicit anywhere.
The proof is strongly related to that in~\cite{Aldous_ExpHittingTimes},
but we wish to stress the relationship between the quantile $t_\eps(A)$
and the exponential approximation, which we will need below.

The second result considers what happens when we have an union of
events
\[
A=A_1\cup A_2\cup\cdots\cup A_\ell.
\]
As described in the \hyperref[sec1]{Introduction}, we give a sufficient condition under
which the hitting time $H_A$ behaves like a minimum of independent exponentials.
%
%
\begin{theorem}\label{thm:rareeventsII}Assume that the set $A$
considered above can be written~as
\[
A = \bigcup_{i=1}^\ell A_i,
\]
where the sets $A_1,\ldots,A_\ell$ are nonempty and
\[
m:=\mathbf{E}_{\pi} [H_{A_1} ] =\mathbf{E}_{\pi}
[H_{A_2} ]=\cdots=\mathbf{E}_{\pi} [H_{A_\ell} ].
\]
Assume $0<\delta<1/5$, $0<\eps<\delta/2\ell$ are such that for all
$1\leq i\leq\ell$,
\[
\forall i\in[\ell] \qquad\mathbf{P}_{\pi} \bigl(H_{A_i}\leq
t_{\mathrm
{mix}}^Q(\delta \eps/2) \bigr)\leq\frac{\delta\eps}{2}.
\]
Then for all $\lambda\in M_1(\bE)$,
\[
\operatorname{Law}_{\lambda} (H_A )=\Expdist \biggl(
\frac
{m}{\ell},2r_\lambda +O (\ell\eps ),O (\delta+\xi ) \biggr),
\]
where
\[
r_\lambda\equiv\mathbf{P}_{\lambda} \bigl(H_A\leq
t_{\mathrm
{mix}}^Q( \delta\eps) \bigr)
\]
and
\[
\xi\equiv\frac{1}{\ell\eps}\sum_{1\leq i<j\leq\ell}\mathbf
{P}_{\pi} (H_{A_i}\leq\eps m, H_{A_j}\leq\eps m ).
\]
\end{theorem}

%
\begin{remark} If the $H_{A_i}$ are in fact independent, then $\xi
=O (\eps\ell )$.
\end{remark}

The remainder of the section is devoted to the proof of these two results.

\subsection{Hitting time of a single set: Proofs}

We first present the proof of \thmref{rareeventsI} modulo two important
lemmas, and subsequently prove those lemmas.

\begin{pf*}{Proof of \thmref{rareeventsI}}
Let $\lambda\in M_1(\bE)$ be arbitrary. Throughout the proof we will
assume implicitly that $\delta+r_\lambda+\eps$ is smaller than some
sufficiently small absolute constant; the remaining case is easy to
handle by increasing the value of $C_0$ if necessary.

We begin with an upper bound for $\mathbf{P}_{\lambda} (H_A\geq
t )$ in terms of
$t_\eps(A)$.
%
%
\begin{lemma}[(Proven in \secref{rareeventupper_proof})]\label
{lem:rareeventupper}Under the assumptions of \thmref{rareeventsI},
\[
\forall t\geq0\qquad \mathbf{P}_{\lambda} (H_A\geq t )\leq
\bigl(1+O (\eps )\bigr) e^{-{\eps(1+O (\delta )) t}/{t_\eps(A)}}.
\]
\end{lemma}
In particular, this implies
%
%
\begin{equation}
\label{eq:expectationupper}\qquad\forall\mu\in M_1(\bE)\qquad
\mathbf{E}_{\mu
} [H_A ] = \int_0^{+\infty}
\mathbf{P}_{\mu} (H_A\geq t ) \,dt\leq\bigl(1+O (\delta )
\bigr) \frac{t_\eps
(A)}{\eps}.
\end{equation}
It turns out that the upper bound in the above lemma can be nearly
reversed if we start from some distribution that is ``far'' from $A$.
%
%
\begin{lemma}[(Proven in \secref{rareeventlower_proof})]\label
{lem:rareeventlower}With the assumptions of \thmref{rareeventsI}, if
$2\eps+r_\lambda<1/2$,
\[
\forall t\geq0\qquad \mathbf{P}_{\lambda} (H_A\geq t )\geq
\bigl(1-O (\eps )-r_\lambda\bigr)_+ e^{-{\eps(1+O
(\delta )) t}/{t_\eps(A)}}.
\]
\end{lemma}
Notice that the combination of these two lemmas already implies the
first statement in the proof, as it shows that for all $t\geq0$,
\begin{eqnarray*}
&&\mathbf{P}_{\lambda} (H_A\geq t ) \\
&&\qquad\in \bigl[\bigl(1-O (\eps
)-2r_\lambda\bigr) e^{-
{\eps t}/{((1+O (\delta )) t_\eps(A))}},\bigl(1+O (\eps )\bigr)
e^{-
{\eps t}/{((1+O (\delta )) t_\eps(A))}} \bigr].
\end{eqnarray*}
To see this, notice that the upper bound is always valid by \lemref
{rareeventupper}. For the lower bound, we use \lemref{rareeventlower}
if $2\eps+r_\lambda\leq1/2$, and note that the lower bound is $0$ if
$2\eps+r_\lambda>1/2$ and the constant in the $O (\eps )$ term
is at
least $4$.

We now prove the assertion about expectations in the theorem. We use
\lemref{approxwasserstein} and deduce
\begin{eqnarray*}
\biggl\llvert\mathbf{E}_{\pi} [H_A ] - \frac{t_\eps
(A)}{\eps
}
\biggr\rrvert& \leq& d_W \bigl(\operatorname{Law}_{\pi}
(H_A ),\Expdist \bigl( \eps^{-1} t_\eps(A)
\bigr) \bigr)
\\
&\leq& O (\delta+r_\pi ) \frac{t_\eps(A)}{\eps},
\end{eqnarray*}
and the assertion follows from dividing by $\eps^{-1} t_\eps(A)$
and noting that
\[
r_\pi=\mathbf{P}_{\pi} \bigl(H_A\leq
t_{\mathrm{mix}}^Q(\delta \eps) \bigr)\leq\delta\eps
\]
by
assumption. The final assertion in the theorem then follows from
\propref{changemean}.
\end{pf*}

\subsubsection{\texorpdfstring{Proof of \protect\lemref{rareeventupper}}{Proof of Lemma 3.2}}\label{sec:rareeventupper_proof}
\mbox{}
\begin{pf}Set $T=t_{\mathrm{mix}}^Q(\delta\eps)$. We note for later reference
that $T<t_\eps(A)$, since
\[
\mathbf{P}_{\pi} (H_A\leq T )\leq\delta\eps<\eps =
\mathbf{P}_{\pi} \bigl(H_A\leq t_\eps(A) \bigr).
\]
Our main goal will be to show the following inequality:
%
%
\begin{equation}
\qquad\forall k\in\N\qquad \mathbf{P}_{\lambda} \bigl(H_A>(k+1)t_\eps
(A) \bigr)\leq (1-\eps+2\delta\eps) \mathbf{P}_{\lambda}
\bigl(H_A>kt_\eps (A) \bigr).
\end{equation}
Once established, this goal will imply
\[
\forall k\in\N\qquad \mathbf{P}_{\lambda} \bigl(H_A\geq
kt_\eps (A) \bigr) \leq(1-\eps+2\delta\eps)^{k}
\]
and
\[
\forall t\geq0\qquad \mathbf{P}_{\lambda} (H_A\geq t )\leq
e^{-\eps(1+O (\delta  )) \lfloor{t}/{t_\eps(A)}
\rfloor} = \bigl(1+O (\eps )\bigr) e^{-{\eps t}/{((1+O (\delta
))t_\eps(A))}},
\]
which is the desired result.
To achieve the goal, we fix some $k\in\N$ and use $T\leq t_\eps(A)$
to bound
\begin{eqnarray*}
\label{eq:tripartition}\mathbf{P}_{\lambda} \bigl(H_A>(k+1)t_\eps
(A) \bigr) &\leq& \mathbf{P}_{\lambda}\pmatrix{
H_A>kt_\eps(A),
\vspace*{2pt}\cr
H_A\circ\Theta_{kt_\eps
(A)+T}>t_\eps(A)-T },
\\
\mbox{(Markov prop.)}&=& \mathbf{P}_{\lambda}
\bigl(H_A>kt_\eps (A) \bigr)\mathbf{P}_{\Lambda}
\bigl(H_A>t_\eps(A)-T \bigr),
\end{eqnarray*}
where $\Lambda$ is the law of $X_{kt_\eps(A)+T}$ conditioned on $\{
H_A>kt_\eps(A)\}$. Since this event belongs to $\sF_{kt_\eps(A)}$ and
$T=t_{\mathrm{mix}}^Q(\delta\eps)$, $\Lambda$ is $\delta\eps
$-close to $\pi
$ in
total variation distance. We deduce
%
%
\begin{equation}
\label{eq:secondver_true}\frac{\mathbf{P}_{\lambda }
(H_A>(k+1)t_\eps(A) )}{\mathbf{P}_{\lambda} (H_A>kt_\eps
(A) )}\leq
\mathbf{P}_{\pi} \bigl(H_A>t_\eps(A)-T \bigr) +
\delta\eps.
\end{equation}
Now observe that
\begin{eqnarray*}
\mathbf{P}_{\pi} \bigl(H_A>t_\eps(A)-T \bigr)
&\leq& \mathbf {P}_{\pi} \bigl(H_A>t_\eps(A)
\bigr)
\\
& &{} + \mathbf{P}_{\pi} \bigl(H_A\in\bigl(t_\eps(A)-T,t_\eps(A)\bigr]
\bigr)
\\
&\leq& \mathbf{P}_{\pi} \bigl(H_A>t_\eps(A)
\bigr)
\\
& & {}+ \mathbf{P}_{\pi} (H_A\circ\Theta_{t_\eps (A)-T}
\leq T ),
\\
\bigl(\mbox{defn. of }t_\eps(A)\bigr) &=& 1-\eps+
\mathbf{P}_{\pi} (H_A\circ\Theta_{t_\eps(A)-T}\leq T ),
\\
\mbox{($\pi$ stationary)} &=& 1-\eps+ \mathbf{P}_{\pi}
(H_A\leq T ),
\\
\bigl(T=t_{\mathrm{mix}}^Q(\delta\eps) + \mbox{assumption}\bigr)&
\leq& 1- \eps+ \delta\eps,
\end{eqnarray*}
and plugging this into \eqnref{secondver_true} gives
\[
\frac{\mathbf{P}_{\lambda} (H_A>(k+1)t_\eps(A)
)}{\mathbf{P}_{\lambda } (H_A>kt_\eps (A) )}\leq \bigl(1-\eps(1-2\delta) \bigr)
\]
as desired.
\end{pf}

\subsubsection{\texorpdfstring{Proof of \protect\lemref{rareeventlower}}{Proof of Lemma 3.3}}\label{sec:rareeventlower_proof}
\mbox{}
\begin{pf}The general scheme of the proof is similar to that of
\lemref{rareeventupper}, but we will need to be a bit more careful in
our estimates. In particular, we will need that $(1+5\delta)\eps<1/2$
and $2\eps+r_\lambda<1/2$.\vadjust{\goodbreak}

Define $T\equiv t_{\mathrm{mix}}^Q(\delta\eps)$ as in the proof of
\lemref
{rareeventupper} in \secref{rareeventupper_proof}. Again observe that
$T<t_\eps(A)$. Define
\[
f(k)\equiv\mathbf{P}_{\lambda} \bigl(H_A\geq
kt_\eps(A) \bigr)\qquad (k\in\N).
\]
Clearly, $f(0)=1$ and
%
%
\begin{eqnarray}
\label{eq:basecase}f(1)&\geq&\mathbf{P}_{\lambda} \bigl(H_A
\circ \Theta_T\geq t_\eps(A) \bigr) -
\mathbf{P}_{\lambda} (H_A \leq T )\geq1-\eps-\delta
\eps-r_\lambda
\nonumber
\\[-8pt]
\\[-8pt]
\nonumber
&\geq&1-2\eps -r_\lambda
\end{eqnarray}
since $T=t_{\mathrm{mix}}^Q(\delta\eps)$, and by the properties of
mixing times,
\[
\mathbf{P}_{\lambda} \bigl(H_A\circ\Theta_T\geq
t_\eps(A) \bigr) \geq\mathbf{P}_{\pi} \bigl(H_A
\geq t_\eps(A) \bigr) - \delta\eps.
\]
We now claim the following:
%
%
\begin{claim}For all $k\in\N\setminus\{0\}$,
\[
\frac{f(k+1)}{f(k)}\geq(1-\eps-5\delta\eps).
\]
\end{claim}
Notice that the claim and \eqnref{basecase} imply
\begin{eqnarray*}
\forall t\geq0\qquad \mathbf{P}_{\lambda} (H_A\geq t )&\geq& f
\bigl( \bigl\lceil t/t_\eps(A) \bigr\rceil \bigr)
\\
&\geq&(1-2\eps-r_\lambda) (1-\eps-5\delta\eps)^{\lceil
{t}/{t_\eps(A)}\rceil-1}
\\
&=&\bigl(1-O (\eps )-r_\lambda\bigr) (1-\eps-5\delta\eps)^{
{t}/{t_\eps(A)}}
\\
&\geq&\bigl(1- O (\eps )-r_\lambda\bigr) e^{-(1+O (\delta ))
{\eps t}/{t_\eps(A)}},
\end{eqnarray*}
which is precisely the bound we wish to prove. We spend the rest of
this proof proving the claim.

Fix some $k\geq1$, and notice that
%
%
\begin{eqnarray}\label{eq:twotermsrarelower}
f(k+1)&\geq& \mathbf{P}_{\lambda} \bigl(H_A\geq
kt_\eps (A), H_A\circ\Theta_{kt_\eps(A)+T}\geq
t_\eps(A)-T \bigr)
\nonumber\\
& &{}- \mathbf{P}_{\lambda} \bigl(H_A\geq
kt_\eps(A), H_A\circ\Theta _{kt_\eps (A)}<T \bigr)
\\
&=:& (I) - (\mathit{II}).\nonumber
\end{eqnarray}
We bound the two terms $(I)$, $(\mathit{II})$ separately. By the Markov property,
\[
(I)= \mathbf{P}_{\lambda} \bigl(H_A\geq kt_\eps(A)
\bigr)\mathbf {P}_{\Lambda} \bigl(H_A\geq t_\eps(A)
\bigr),
\]
where $\Lambda$ is the conditional law of $X_{kt_\eps(A)+T}$ given
$H_A\geq kt_\eps(A)$. Since $T=t_{\mathrm{mix}}^Q(\eps\delta)$,
$\Lambda$ is
within distance $\delta\eps$ from $\pi$. We deduce
%
%
\begin{equation}
\label{eq:firsttermrarelower}\quad (I)\geq\mathbf{P}_{\lambda}
\bigl(H_A\geq kt_\eps(A) \bigr) \bigl(
\mathbf{P}_{\pi} \bigl(H_A\geq t_\eps(A) \bigr)-
\delta\eps \bigr) = f(k) (1-\eps-\delta\eps).
\end{equation}
We now upper bound term $(\mathit{II})$ in \eqnref{twotermsrarelower}. Notice
that (again because of the Markov property)
\[
(\mathit{II})\leq\mathbf{P}_{\pi} \bigl(H_A\geq(k-1)t_\eps(A),
H_A\circ \Theta_{kt_\eps(A)}<T \bigr) = f(k-1)
\mathbf{P}_{\Lambda'} (H_A< T ),
\]
where $\Lambda'$ is the law of $X_{kt_\eps(A)}$ conditioned on $\{
H_A\geq(k-1)t_\eps(A)\}$. Recalling that $t_\eps(A)\geq T =
t_{\mathrm{mix}}
^Q(\delta\eps)$, we see that $\Lambda'$ is $\delta\eps$-close to
$\pi$.
Since we have also assumed that $\mathbf{P}_{\pi} (H_A\leq
T )\leq\delta\eps$,
we deduce
\[
(\mathit{II})\leq f(k-1) \bigl(\mathbf{P}_{\pi} (H_A< T )+\delta
\eps \bigr)\leq2\delta\eps f(k-1).
\]
We combine this with \eqnref{firsttermrarelower} and \eqnref
{twotermsrarelower} to obtain
\[
\forall k\in\N\setminus\{0,1\} f(k+1)\geq f(k) (1-\eps- \delta \eps) - f(k-1)
(2\delta\eps).
\]
One can argue inductively that $f(k)/f(k-1)\geq1/2$ for all $k\geq1$.
Indeed, this holds for $k\geq2$ by the claim applied to $k-1$. For
$k=1$ we may use \eqnref{basecase} and the assumption on $2\eps
+r_\lambda$ to deduce the same result. Applying this to the previous
inequality, we obtain
\[
\forall k\in\N\setminus\{0\} f(k+1)\geq f(k) (1-\eps- 5\delta\eps),
\]
which finishes the proof of the claim and of the lemma.
\end{pf}

\subsection{Hitting times of a union of sets: Proofs}

We present the proof of \thmref{rareeventsII} below.
\begin{pf*}{Proof of \thmref{rareeventsII}}
There are three main steps in the proof, here outlined in a slightly
oversimplified way:
\begin{longlist}[(1)]
\item[(1)] We show that \thmref{rareeventsI} is applicable to the
hitting times of $A_1, \ldots,\break A_\ell$. In particular, this shows
that $\mathbf{P}_{\pi} (H_{A_i}\leq\eps m )\approx\eps$.
\item[(2)] We show that
\[
\mathbf{P}_{\pi} (H_A\leq\eps m ) \approx\sum
_{i=1}^\ell \mathbf{P}_{\pi}
(H_{A_i}\leq\eps m )\approx\ell\eps,
\]
so that $t_{\ell\eps}(A)\approx\eps m$.
\item[(3)] Finally, we apply \thmref{rareeventsI} to $H_A$ and deduce
that this random variable is approximately exponential with mean
\[
\mathbf{E}_{\pi} [H_A ] \approx t_{\eps\ell}(A)/\eps
\ell\approx m/ \ell.
\]
\end{longlist}

The actual proof is only slightly more complicated than this outline.
We begin with a claim corresponding to step $1$ above.
%
%
\begin{claim}\label{claim:quantilesaregood}For all $1\leq i\leq\ell$,
\[
\eps_i\equiv\mathbf{P}_{\pi} (H_{A_i}\leq\eps m
)=\bigl(1+O (\delta )\bigr) \eps.
\]
\end{claim}
\begin{pf}
Consider some $\eps'\in[\eps/2,2\eps
]$. Notice that
\[
t_{\mathrm{mix}}^Q \bigl(\delta\eps' \bigr)\leq
t_{\mathrm{mix}}^Q( \delta\eps/2),
\]
and therefore,
\[
\mathbf{P}_{\pi} \bigl(H_{A_i}\leq t_{\mathrm{mix}}^Q
\bigl(\delta \eps' \bigr) \bigr)\leq\frac{\delta\eps
}{2}\leq\delta
\eps'.
\]
This shows that \thmref{rareeventsI} is applicable with $A_i$ replacing
$A$ and $\eps'$ replacing $\eps$. We deduce in particular that
\[
\forall\frac{\eps}{2}\leq\eps'\leq2\eps\qquad \biggl\llvert
\frac{\eps
'\mathbf{E}_{\pi } [H_{A_i} ]}{t_{\eps'}(A_i)}-1 \biggr\rrvert\leq O \bigl(\delta+\eps' \bigr) =
O (\delta ).
\]
In particular, there exists a universal constant $c>0$ such that if
$\eps'\leq(1-c\delta) \eps$, then
$t_{\eps'}(A_i)<\eps\mathbf{E}_{\pi} [H_{A_i} ]$,
whereas if
$\eps'>(1+c\delta) \eps$,
$t_{\eps'}(A_i)>\eps\mathbf{E}_{\pi} [H_{A_i} ]$. In
other words,
\[
(1-c\delta) \eps\leq\mathbf{P}_{\pi} \bigl(H_{A_i}\leq\eps
\mathbf{E}_{\pi} [H_{A_i} ] \bigr)\leq(1+c\delta)\eps.
\]
\upqed\end{pf}
We now come to the second part of the proof.
%
%
\begin{claim}\label{claim:etaishere}Let $\xi$ be as in the statement of
\thmref{rareeventsII}. Then
\[
\mathbf{P}_{\pi} (H_A\leq\eps m ) = \bigl(1+O (\delta +
\xi )\bigr) \ell\eps.
\]
In particular, there exists a number $\eta=(1+O (\delta+\xi
))\ell
\eps$ with
$\eps m = t_\eta(A)$.
\end{claim}
\begin{pf}To see this, we note that
\[
\{H_A\leq\eps m\} = \bigcup_{i=1}^\ell
\{H_{A_i}\leq\eps m\}.
\]
The union bound gives
\[
\mathbf{P}_{\pi} (H_A\leq\eps m )\leq\sum
_{i=1}^\ell \mathbf{P}_{\pi}
(H_{A_i}\leq\eps m )\leq\bigl(1+O (\delta )\bigr) \ell\eps.
\]
A lower bound can be obtained via the Bonferroni inequality,
\begin{eqnarray*}
\mathbf{P}_{\pi} (H_A\leq\eps m )&\geq&\sum
_{i=1}^\ell \mathbf{P}_{\pi}
(H_{A_i}\leq\eps m )
\\
& &{} - \sum_{1\leq i<j\leq\ell}\mathbf{P}_{\pi}
(H_{A_i}\leq \eps m,H_{A_j}\leq\eps m )
\\
&=& \bigl(1+O (\delta+\xi )\bigr) \ell\eps,
\end{eqnarray*}
using the definition of $\xi$.\vadjust{\goodbreak}
\end{pf}
We now need to show that
the assumptions of \thmref{rareeventsI} are applicable to~$H_A$,
with the value of $\eta$ in \claimref{etaishere} replacing $\eps$. We
assume that $\delta+\xi$ is small enough, which we may do because
otherwise the theorem is trivial. In particular, we can assume that the
$O (\xi+\delta )$ term in the expression for $\eta$ is between
$-1/2$ and $1$, so that
\[
\frac{\eps\ell}{2}\leq\eta\leq2\eps\ell.
\]
Since we also assumed $\eps<\delta/2\ell$, we have $\eta<\delta$.
Moreover, $t_{\mathrm{mix}}^Q(\delta\eta)\leq
t_{\mathrm{mix}}^Q(\delta\eps/2)$. This implies
\[
\mathbf{P}_{\pi} \bigl(H_A\leq t_{\mathrm{mix}}^Q(
\delta\eta ) \bigr)\leq\sum_{i=1}^\ell
\mathbf{P}_{\pi} \bigl(H_{A_i}\leq t_{\mathrm
{mix}}^Q(
\delta\eps/2) \bigr)\leq\ell\delta\eps/2\leq\delta \eta.
\]
We may now apply \thmref{rareeventsI} (with $\eta$ replacing $\eps$) to
deduce that for any $\lambda\in M_1(\bE)$,
\[
\operatorname{Law}_{\lambda} (H_A )=\Expdist
\bigl(t_\eta (A)/\eta, O (\eta )+2r_\lambda,O (\delta+\xi )
\bigr).
\]
To finish the proof, we note that $\eta=O (\ell\eps )$,
\[
t_\eta(A)/\eta= \eps m/\bigl(1+O (\delta+\xi )\bigr) \ell\eps=
\bigl(1+O (\delta+\xi )\bigr) \frac{m}{\ell}
\]
and apply \propref{changemean}.
\end{pf*}

\section{Meeting times of multiple random walks}\label{sec:meeting}

We now put our two exponential approximation results to use, showing
that the meeting times we are interested in are well approximated by
exponential random variables. Much of the work needed for this is
contained in technical estimates whose proofs can be safely skipped in
a first reading.

\subsection{Basic definitions}\label{sec:setup}

For the remainder of this section, $\bV$ is a finite set, and~$Q$ is
the generator of a mixing Markov chain over
$\bV$ with mixing times $t_{\mathrm{mix}}^Q(\cdot)$ and stationary measure
$\pi$. For each $k\in\N\setminus\{0,1\}$ we will also consider the
Markov chains $Q^{(k)}$ over $\bE^k$
that correspond to $k$ independent realizations of $Q$ from
prescribed initial states, as defined in \secref{productchains}. We
will also follow the notation from that section.

For $k=2$, we define the first meeting time
%
%
\begin{equation}
\label{eq:defM}M\equiv\inf \bigl\{t\geq0 \dvtx X_t(1)=X_t(2)
\bigr\}
\end{equation}
and the parameters
%
%
\begin{eqnarray}
\label{eq:defm}\mathsf{m}(Q)&\equiv& \mathbf{E}_{\pi^{\otimes
2}} [M ],
\\
\label{eq:defratio}\rho(Q)&\equiv& \frac{t_{\mathrm
{mix}}^Q}{\mathsf{m}(Q)}.
\end{eqnarray}

We also define an extra prameter $\mathsf{err}(Q)$ which will appear
as an error
term at several different points in the paper. Ths parameter $\mathsf
{err}(Q)$ is
defined as
%
%
\begin{equation}
\label{eq:deferr}\qquad \mathsf{err}(Q)=c_0 \sqrt{\rho(Q)\ln\bigl(1/\rho
(Q)\bigr)}\qquad\mbox{if $Q$ is reversible and transitive.}
\end{equation}
For other $Q$, we define it as
%
%
\begin{equation}
\label{eq:deferrtwo}\mathsf{err}(Q)= c_1 \sqrt{
\bigl(1+q_{\max
}t_{\mathrm{mix}} ^Q \bigr) \pi_{\max}
\ln \biggl(\frac{1}{(1+q_{\max}t_{\mathrm
{mix}}^Q) \pi
_{\max
}} \biggr)}.
\end{equation}
The numbers $c_0,c_1>0$ are universal constants that we do not specify
explicitly. We choose them so as to satisfy Propositions \ref
{prop:goodbound},~\ref{prop:goodboundgeneral} and \ref
{prop:correlgeneral} below.

We now take $k>2$ and consider the process $Q^{(k)}$, with trajectories
\[
\bigl(X_t^{(k)}= \bigl(X_t(1),X_t(2),
\ldots,X_t(k) \bigr) \bigr)_{t\geq0}
\]
corresponding to $k$ independent realizations of $Q$; cf. \secref
{productchains}. This has stationary
distribution $\pi^{\otimes k}$.

We write $M^{(k)}$ for the first meeting time among these random walks,
%
%
\begin{equation}
\label{eq:defM}M^{(k)}\equiv\inf \bigl\{t\geq0 \dvtx \exists1\leq i<j\leq k,
X_t(i)=X_t(j) \bigr\}.
\end{equation}
One may note that
\[
M^{(k)} = \min_{\{i,j\}\in({[k]\atop 2})}M_{i,j},
\]
where $\bigl({[k]\atop2}\bigr)$ was defined in \secref{prelim}, and for $1\leq
i<j\leq k$,
%
%
\begin{equation}
\label{eq:defMij}M_{i,j} = M_{j,i}\equiv\inf \bigl\{t\geq0 \dvtx X_t(i)=X_t(j) \bigr\}
\end{equation}
is distributed as $M$ for a realization of $Q^{(2)}$ starting from
$(X_0(i),X_0(j))$.

\subsection{Technical estimates for reversible and transitive chains}

In this subsection we collect the estimates that we will use in the
case of chains that are reversible and transitive.
%
%
\begin{proposition}\label{prop:goodbound}Assume $Q$ is reversible and
transitive and define $\mathsf{err}(Q)$ accordingly. If $\mathsf
{err}(Q)\leq1/4$, then
\[
\mathbf{P}_{\pi^{\otimes2}} \bigl(M\leq t_{\mathrm{mix}}^Q \bigl(
\mathsf{err}(Q)^2 \bigr) \bigr)\leq\mathsf{err}(Q)^2.
\]
\end{proposition}
%
%
\begin{remark} The proof is entirely general, but we will only use
this estimate in the transitive/reversible case.
\end{remark}
\begin{pf}We will prove a result in contrapositive form: if $0<\beta
<1/4$ is such that
\[
\mathbf{P}_{\pi^{\otimes 2}} \bigl(M\leq t_{\mathrm{mix}}^Q(\beta )
\bigr)>\beta,
\]
then $\beta<c^2_0\rho(Q)\ln(1/\rho(Q))$ for some universal $c_0>0$.\vadjust{\goodbreak}

Notice that for any $x^{(2)}\in\bV^2$,
\begin{eqnarray*}
\mathbf{P}_{x^{(2)}} \bigl(M>t_{\mathrm{mix}}^Q(\beta/4) +
t_{\mathrm{mix}}^Q( \beta) \bigr) &\leq&\mathbf{P}_{x^{(2)}}
\bigl(M\circ\Theta_{t_{\mathrm{mix}}^Q(\beta/4)}> t_{\mathrm{mix}} ^Q(\beta) \bigr)
\\
&\leq&\mathbf{P}_{\pi^{\otimes 2}} \bigl(M>t_{\mathrm
{mix}}^Q(\beta)
\bigr) +\beta/2
\\
&<& (1-\beta/2),
\end{eqnarray*}
where the middle inequality follows from the fact that $t_{\mathrm
{mix}}^Q(\beta
/4)$ is an upper bound for the $\beta/2$-mixing time of $Q^{(2)}$; cf.
\lemref{productchains}. A standard argument using the Markov property
implies that for any $k\in\N$,
\[
\mathbf{P}_{x^{(2)}} \bigl(M>k \bigl(t_{\mathrm{mix}}^Q(
\beta/4) + t_{\mathrm{mix}} ^Q(\beta) \bigr) \bigr)<(1-
\beta/2)^k,
\]
so that
\[
\mathsf{m}(Q)= \mathbf{E}_{\pi^{\otimes2}} [M ]\leq C \frac{t_{\mathrm{mix}}^{Q}(\beta)
+t_{\mathrm{mix}}
^{Q}(\beta/4)}{\beta}.
\]

Since $t_{\mathrm{mix}}^Q(\alpha)\leq C\ln(1/\alpha) t_{\mathrm
{mix}}^Q$, we deduce that
\[
\frac{\beta}{c \ln(1/\beta)}<\rho(Q),
\]
with $c>0$ universal, which implies the desired result.
\end{pf}

We now prove an estimate on correlations.

%
\begin{proposition}\label{prop:correltransitive}Assume $Q$ is
transitive. Then for all $t,s\geq0$ and $\{i,j\},\{\ell,r\}\subset
\bV
$ with $\{i,j\}\neq\{r,\ell\}$,
\[
\mathbf{P}_{\pi^{\otimes k}} (M_{i,j}\leq t, M_{\ell,r}\leq s )
\leq2 \mathbf{P}_{\pi^{\otimes2}} (M\leq s )\mathbf{P}_{\pi
^{\otimes2}} (M\leq t ).
\]
\end{proposition}
\begin{pf}If $\{i,j\}\cap\{\ell,r\}=\varnothing$, the events $\{
M_{i,j}\leq t\}$ and $\{M_{\ell,r}\leq t\}$ are independent. Since the
laws of both $M_{i,j}$ and $M_{\ell,r}$ under $\pi^{\otimes k}$ are
equal to the law of $M$ under $\pi^{\otimes2}$, we obtain
\[
\mathbf{P}_{\pi^{\otimes k}} (M_{i,j}\leq t, M_{\ell,r}\leq s )
\leq \mathbf{P}_{\pi^{\otimes2}} (M\leq s )\mathbf{P}_{\pi
^{\otimes2}} (M\leq t )
\]
in this case.
Assume now $\{i,j\}\cap\{\ell,r\}$ has one element. Without loss of
generality we may assume $k=3$, $\{i,j\}=\{1,2\}$ and $\{\ell,r\}=\{
1,3\}$. We have
%
%
\begin{eqnarray}\label{eq:cortranS}
\mathbf{P}_{\pi^{\otimes3}} (M_{1,2}\leq t, M_{1,3}
\leq s )&\leq&\mathbf{P}_{\pi^{\otimes3}} (M_{1,2}\leq
t,M_{1,3}\circ \Theta_{M_{1,2}}\leq s )
\nonumber
\\[-8pt]
\\[-8pt]
\nonumber
& & {} + \mathbf{P}_{\pi^{\otimes 3}} (M_{1,3}\leq
s,M_{1,2}\circ\Theta_{M_{1,3}}\leq t ).
\end{eqnarray}
Consider the first term on the RHS. By the Markov property,
\[
\mathbf{P}_{\pi^{\otimes3}} (M_{1,2}\leq t,M_{1,3}\circ\Theta
_{M_{1,2}}\leq s ) = \mathbf{P}_{\pi^{\otimes2}} (M\leq t )
\mathbf{P}_{\lambda^{(2)}} (M\leq s ),
\]
where $\lambda^{(2)}$ is the law of $X_{M_{1,2}}(1),X_{M_{1,2}}(3)$
conditionally on $M_{1,2}\leq t$. Since $(X_t(3))_t$ is stationary and
independent from\vadjust{\goodbreak} this event, $\lambda^{(2)}=\lambda\otimes\pi$ for
some $\lambda\in M_1(\bV)$ which is the law of $X_{M_{1,2}}(1)$ under
$\mathbf{P}_{\pi^{\otimes2}}$. The transitivity of $Q$ (which
implies that
$\pi$ is uniform) implies that $\lambda=\pi$ and therefore
\[
\mathbf{P}_{\pi^{\otimes3}} (M_{1,2}\leq t,M_{1,3}\circ\Theta
_{M_{1,2}}\leq s ) = \mathbf{P}_{\pi^{\otimes2}} (M\leq t )
\mathbf{P}_{\pi^{\otimes2}} (M\leq s ).
\]
The same bound can be shown for the other term in the RHS of \eqnref
{cortranS}, and this implies the proposition.
\end{pf}

\subsection{Technical estimates for the general case}

We will need the following general result:

%
\begin{proposition}\label{prop:arbitrarymeeting}For any $\lambda\in
M_1(\bV)$ and $T\geq0$,
\[
\mathbf{P}_{\lambda\otimes\pi} (M\leq T )\leq(1+2T q_{\max})
\pi_{\max}.
\]
\end{proposition}
\begin{pf}Let $(X_t)_t$ be a single realization of $Q$. One may
imagine that the trajectory of $(X_t)_{t\geq0}$ is sampled as follows.
First, let $\sP$ be a Poisson process with intensity $q_{\max}$
independent from the initial state $X_0$. At each time $t\in\sP$, one
updates the value of $X_t$ as follows: if $X_s=x$ for $s$ immediately
before $t$, one sets
\[
X_t = y \mbox{ with probability}\qquad\frac{q(x,y)}{q_{\max}} \bigl(y\in \bV
\setminus\{x\} \bigr)
\]
and $X_t=x$ with the remaining probability. This implies that, at the
points of the Poisson process, $X_t$ is updated as in the discrete-time
Markov chain with matrix $P=(I + Q/q_{\max})$, and it is easy to see
that $\pi$ is stationary for this chain.

Now let $X_t(1),X_t(2)$ be independent trajectories of $Q$, with
$X_t(1)$ started from~$\lambda$ and $X_t(2)$ started from the
stationary distribution $\pi$. We will imagine that each $X_t(i)$ has
its own Poisson process $\sP(i)$ and was generated in the way described
above. It then follows that
\begin{eqnarray*}
\mathbf{P}_{\lambda\otimes\pi} \bigl(M\leq T\mid\sP(1),\sP (2) \bigr) &\leq&
\mathbf{P}_{\lambda\otimes\pi} \bigl(X_0(1)=X_0(2) \bigr)
\\
& & {}+ \sum_{t\in
\sP
}\mathbf{P}_{\lambda\otimes\pi}
\bigl(X_t(1)=X_t(2)\mid\sP(1),\sP (2) \bigr),
\end{eqnarray*}
where $\sP= \sP(1)\cup\sP(2)$, since the processes can only change
values at the times of the two Poisson processes. At time $0$, we have
\[
\mathbf{P}_{\lambda\otimes\pi} \bigl(X_0(1)=X_0(2) \bigr)
= \sum_{x\in\bV}\lambda(x)\pi(x)\leq\sum
_{x\in\bV}\lambda(x)\pi_{\max}\leq\pi_{\max}.
\]
For $t\in\sP$, the law of $X_t(1),X_t(2)$ equals
\[
\bigl(\lambda P^{k_1} \bigr)\otimes \bigl(\pi P^{k_2} \bigr),
\]
where $k_i=|\sP(i)\cap(0,t]|$ ($i=1,2$). Crucially, $\pi$ is
stationary for $P$, hence $\pi P^{k_2}=\pi$, and we obtain
\[
\mathbf{P}_{\lambda\otimes\pi} \bigl(X_t(1)=X_t(2)\mid
\sP(1), \sP (2) \bigr) = \sum_{x\in\bV
} \bigl(\lambda
P^{k_1} \bigr) (x)\pi(x)\leq\pi_{\max}\vadjust{\goodbreak}
\]
as for $t=0$. We deduce
\[
\mathbf{P}_{\lambda\otimes\pi} \bigl(M\leq T\mid\sP(1),\sP (2) \bigr) \leq\bigl(1 + \bigl|
\bigl(\sP(1)\cup \sP(2) \bigr)\cap(0,T]\bigr|\bigr) \pi_{\max}.
\]
The proposition follows from taking expectations on both sides and
noticing that
\[
\mathbf{E} \bigl[\bigl| \bigl(\sP(1)\cup\sP(2) \bigr)\cap(0,T]\bigr| \bigr] = 2T
q_{\max}.
\]
\upqed\end{pf}

We now prove an estimate corresponding to \propref{goodbound} in this
general setting.

%
\begin{proposition}\label{prop:goodboundgeneral}Assume $\mathsf
{err}(Q)$ is as
defined in \eqnref{deferrtwo}. Then
\[
\mathbf{P}_{\pi^{\otimes2}} \bigl(M\leq t_{\mathrm{mix}}^Q \bigl(
\mathsf{err}(Q)^2 \bigr) \bigr)\leq\mathsf{err}(Q)^2.
\]
\end{proposition}
\begin{pf}The previous proposition implies
\begin{eqnarray*}
\mathbf{P}_{\pi^{\otimes2}} \bigl(M\leq t_{\mathrm{mix}}^Q \bigl(
\mathsf{err}(Q)^2 \bigr) \bigr)&\leq& \bigl(1+2t_{\mathrm
{mix}}^Q
\bigl(\mathsf{err}(Q)\bigr) q_{\max} \bigr) \pi_{\max}
\\
&\leq&C \bigl(1+2t_{\mathrm{mix}}^Q q_{\max} \bigr)
\pi_{\max} \ln\bigl(1/\mathsf{err}(Q)\bigr).
\end{eqnarray*}
This is $\leq\mathsf{err}(Q)^2$ by definition of this quantity, if we choose
$c_1$ in \eqnref{deferrtwo} to be large enough.
\end{pf}
We now prove an estimate on correlations that is similar to \propref
{correltransitive}, but with an extra term. Recall that
$\bigl({[k]\atop2}\bigr)$ was defined in \secref{notation}.

%
\begin{proposition}\label{prop:correlgeneral}For any mixing Markov
chain $Q$, if one defines $\mathsf{err}(Q)$ as in \eqnref{deferrtwo},
we have
the following inequality for $k\geq3$ and all distinct pairs $\{i,j\},\{\ell,r\}\in\bigl({[k]\atop2}\bigr)$:
\[
\mathbf{P}_{\pi^{\otimes k}} (M_{i,j}\leq t,M_{\ell,r}\leq s )
\leq2 \mathbf{P}_{\pi^{\otimes2}} (M\leq t )\mathbf{P}_{\pi
^{\otimes2}} (M\leq s )+O
\bigl(\mathsf{err}(Q)^2 \bigr).
\]
\end{proposition}
\begin{pf}The case $\{i,j\}\cap\{\ell,r\}=\varnothing$ follows as in
the proof of \propref{correltransitive}. In case $\{i,j\}\cap\{\ell,r\}
$ has one element, we may again assume that $i=\ell=1$, $j=2$ and
$k=3$. Equation \eqnref{cortranS} still applies, so we proceed to bound
\[
\mathbf{P}_{\pi^{\otimes3}} (M_{1,2}\leq t,M_{1,3}\circ\Theta
_{M_{1,2}}\leq s ),
\]
which is upper bounded by
%
%
\begin{eqnarray*}
&&\mathbf{P}_{\pi^{\otimes3}} (M_{1,2}\leq t,M_{1,3}
\circ \Theta _{M_{1,2}}\leq s )
\\
&&\qquad\leq\mathbf{P}_{\pi^{\otimes 3}} \bigl(M_{1,2}\leq t,M_{1,3}
\circ \Theta_{M_{1,2}}\leq t_{\mathrm{mix}}^Q(\eta) \bigr)
\\
&&\qquad\quad{}+ \mathbf{P}_{\pi^{\otimes3}} (M_{1,2}\leq t,M_{1,3}\circ
\Theta _{M_{1,2}+t_{\mathrm{mix}} ^Q(\eta)}\leq s ) =(I)+(\mathit{II})
\end{eqnarray*}
for some $\eta\in(0,1/4)$ to be chosen later.\vadjust{\goodbreak}

Term $(I)$ is equal to
\[
\mathbf{P}_{\pi^{\otimes3}} (M_{1,2}\leq t )\mathbf {P}_{\lambda^{(2)}}
\bigl(M\leq t_{\mathrm{mix}}^Q(\eta) \bigr),
\]
where $\lambda^{(2)}$ is the law of $(X_{M_{1,2}}(2),X_{M_{1,2}}(3))$
conditionally on $\{M_{1,2}\leq t\}$. As in the previous proof,
$(X_t(3))_t$ is stationary and independent from the conditioning, hence
$\lambda^{(2)}=\lambda\otimes\pi$ for some $\lambda\in M_1(\bV)$. We
use \propref{arbitrarymeeting} to deduce
\[
(I)\leq\mathbf{P}_{\pi^{\otimes3}} (M_{1,2}\leq t ) O \bigl(
\bigl(1+t_{\mathrm{mix}} ^Q(\eta) q_{\max} \bigr) \pi
_{\max} \bigr).
\]
The analysis of term $(\mathit{II})$ is simpler: we have
\[
(\mathit{II})=\mathbf{P}_{\pi^{\otimes3}} (M_{1,2}\leq t )\mathbf
{P}_{\lambda_*\otimes\pi } (M\leq s )
\]
for some $\lambda_*\in M_1(\bV)$ which is the law of
$X_{M_{1,2}+t_{\mathrm{mix}}
^Q(\eta)}$ conditionally on $\{M_{1,2}\leq t\}$. The time shift by
$t_{\mathrm{mix}}^Q(\eta)$ implies that $\lambda_*$ is $\eta$-close to
stationary, hence
\[
(\mathit{II})\leq\mathbf{P}_{\pi^{\otimes3}} (M_{1,2}\leq t ) \bigl(\eta+
\mathbf{P}_{\pi ^{\otimes 2}} (M\leq s ) \bigr).
\]
We deduce that
\begin{eqnarray*}
&&\mathbf{P}_{\pi^{\otimes3}} (M_{1,2}\leq t,M_{1,3}\circ\Theta
_{M_{1,2}}\leq s )
\\
&&\qquad\leq\mathbf{P}_{\pi^{\otimes2}} (M\leq t ) \mathbf {P}_{\pi^{\otimes2}} (M\leq s )
\\
&&\qquad\quad{}+ \mathbf{P}_{\pi^{\otimes2}} (M\leq t ) O \bigl(\eta+ \bigl(1+t_{\mathrm{mix}}^Q(
\eta) q_{\max} \bigr) \pi_{\max
} \bigr).
\end{eqnarray*}
Recall $t_{\mathrm{mix}}^Q(\eta)\leq C t_{\mathrm{mix}}^Q \ln
(1/\eta)$ for some universal
$C>0$. If
\[
\eta_0 \equiv \bigl(1+q_{\max} t_{\mathrm{mix}}^Q
\bigr) \pi_{\max}\leq1/2,
\]
we may take $\eta=\eta_0$ to obtain
\begin{eqnarray*}
\mathbf{P}_{\pi^{\otimes3}} (M_{1,2}\leq t,M_{1,3}\circ\Theta
_{M_{1,2}}\leq s )& \leq& \mathbf{P}_{\pi^{\otimes2}} (M\leq t )
\mathbf{P}_{\pi ^{\otimes2}} (M\leq s )
\\
& &{}+ \mathbf{P}_{\pi^{\otimes2}} (M\leq t ) O \biggl(\eta \ln \biggl(
\frac{1}{\eta} \biggr) \biggr).
\end{eqnarray*}
The case of $\eta_0\geq1/2$ is covered ``automatically'' by the big-oh
notation.

An analogous bound can be obtained with the roles of $(t,2)$ and
$(s,3)$ reversed. Plugging these into \eqnref{cortranS} gives the
desired bound.
\end{pf}

\subsection{Exponential approximation for a pair of particles}\label
{sec:pairofparticles}

We now come back to the setting of \secref{setup} and show $M$ is
approximately exponentially distributed.

%
\begin{lemma}\label{lem:firstmeetingtime}Define $\mathsf{err}(Q)$ as
in \eqnref
{deferr} (if $Q$ is reversible and transitive) or as in \eqnref
{deferrtwo} (if not). Then $\forall\lambda^{(2)}\in M_1(\bV^{(2)})$
\[
\operatorname{Law}_{\lambda^{(2)}} (M ) = \Expdist \bigl(\mathsf{m}(Q),O \bigl(
\mathsf{err}(Q) \bigr) + 2r\lambda ^{(2)},O \bigl(\mathsf{err}(Q) \bigr)
\bigr),
\]
where
\[
r_{\lambda^{(2)}}=\mathbf{P}_{\lambda^{(2)}} \bigl(M\leq t_{\mathrm
{mix}}^Q
\bigl( \mathsf{err}(Q)^2 \bigr) \bigr).
\]
\end{lemma}

\begin{pf}This is a direct application of \thmref{rareeventsI} to
the hitting time of the diagonal set
\[
\Delta\equiv \bigl\{(x,x) \dvtx x\in\bV \bigr\}\subset\bV^2
\]
by the chain with generator $Q^{(2)}$ defined in \secref{prelim} and
with $\eps=\mathsf{err}(Q)$, $\delta=2\mathsf{err}(Q)$. All we need
to show is that
\[
\mathbf{P}_{\pi^{\otimes2}} \bigl(M\leq t_{\mathrm
{mix}}^{Q^{(2)}}(\delta
\eps) \bigr)\leq\eps\delta,
\]
where $t_{\mathrm{mix}}^{Q^{(2)}}(\cdot)$ denotes the mixing times of
$Q^{(2)}$.
This inequality follows from
\[
t_{\mathrm{mix}}^{Q^{(2)}} \bigl(2\mathsf{err}(Q)^2 \bigr)\leq
t_{\mathrm{mix}}^Q \bigl( \mathsf{err}(Q)^2 \bigr)\mbox{
\qquad(\lemref{productchains})}
\]
and
\[
\mathbf{P}_{\pi^{\otimes2}} \bigl(M\leq t_{\mathrm{mix}}^Q \bigl(
\mathsf{err}(Q)^2 \bigr) \bigr)\leq\mathsf{err}(Q)^2<
\delta\eps,
\]
which follows from \propref{goodbound} in the reversible/transitive
case and Proposition~\ref{prop:goodboundgeneral} in the general case.
\end{pf}

\subsection{Exponential approximation for many random walkers}\label
{sec:severalparticles}

We now consider the more complex problem of bounding the meeting times
among $k\geq2$ particles. We take the notation in \secref{setup} for granted.

%
\begin{lemma}\label{lem:firstmeetingtimemultiple}Let $\ell=
\bigl({k\atop2}\bigr)>0$, and assume that the quantity $\mathsf{err}(Q)$
defined in
\eqnref
{deferr} (if $Q$ is reversible and transitive) or as in \eqnref
{deferrtwo} (if not) satisfies $\mathsf{err}(Q)\leq1/10\ell$. Then
for all
$\lambda^{(k)}\in M_1(\bV^k)$,
\[
\operatorname{Law}_{\lambda^{(k)}} \bigl(M^{(k)} \bigr) = \Expdist
\biggl( \frac{\mathsf{m}(Q)}{\ell
},O \bigl(k^2\mathsf{err}(Q) \bigr) +
2r_{\lambda^{(k)}}, O \bigl(k^2\mathsf{err}(Q) \bigr) \biggr),
\]
where $r_{\lambda^{(k)}} = \mathbf{P}_{\lambda^{(k)}}
(M^{(k)}\leq t_{\mathrm{mix}} ^Q(\mathsf{err}(Q)^2) )$.
\end{lemma}
\begin{pf}$M^{(k)}$ is the hitting time of a union of $\ell$ sets:
\[
\Delta^{(k)} \equiv\bigcup_{\{i,j\}\in({k\atop 2})}
\Delta_{\{i,j\}
}\qquad\mbox{where }\Delta_{\{i,j\}}\equiv \bigl
\{x^{(k)}\in\bV^k \dvtx x^{(k)}(i)=x^{(k)}(j)
\bigr\}.
\]
We will apply \thmref{rareeventsII}, applied to the product chain
$Q^{(k)}$, to show that this hitting time is approximately exponential.
We set $\delta=2\ell\mathsf{err}(Q)$, $\eps=\mathsf{err}(Q)$ and
verify the conditions
of the theorem:\vadjust{\goodbreak}
\begin{itemize}
\item$0<\delta<1/5$, $0<\eps<\delta/2\ell$: These conditions follow
from $\mathsf{err}(Q)<1/10\ell$.
\item$\mathbf{P}_{\pi^{\otimes k}} (M_{i,j}\leq t_{\mathrm
{mix}}^{Q^{(k)}}(\delta\eps /2) )\leq\delta\eps/2$. To prove
this we simply observe that
\[
t_{\mathrm{mix}}^{Q^{(k)}}(\delta\eps/2)\leq t_{\mathrm
{mix}}^{Q}
\bigl( \mathsf{err}(Q)^2 \bigr)\qquad\mbox{(\lemref{productchains} and
defn. of $\eps, \delta$)}
\]
and that
\[
\mathbf{P}_{\pi^{\otimes2}} \bigl(M\leq t_{\mathrm{mix}}^{Q} \bigl(
\mathsf{err}(Q)^2 \bigr) \bigr)\leq\mathsf{err}(Q)^2 =
\frac
{\delta\eps}{2\ell}\leq\frac{\delta\eps}{2}
\]
by \propref{goodbound} (in the reversible/transitive case) or by
\propref{goodboundgeneral} (in general).
\item$\mathbf{E}_{\pi^{\otimes k}} [H_{\Delta_{\{i,j\}}}
]=\mathsf{m}(Q)$ is the same
for all $\{i,j\}\in\bigl({[k]\atop2}\bigr)$: this is obvious.
\end{itemize}
The lemma will then follow once we show that the $\xi$ quantity in
\thmref{rareeventsII}, which in this case equals
\[
\xi= \sum_{\{i,j\}\neq\{\ell,r\}\ \mathrm{in}\ ({[k]\atop
2})}\frac
{\mathbf{P}_{\pi^{\otimes k}} (M_{\{i,j\}}\leq\eps\mathsf
{m}(Q), M_{\{\ell,r\} }\leq \eps\mathsf{m}(Q) )}{\ell\eps},
\]
and satisfies $\xi=O (k^2\mathsf{err}(Q) ).$ To start, we
go back to
\claimref
{quantilesaregood} in the proof of \thmref{rareeventsII} and observe
that whenever the assumptions of that theorem hold,
%
%
\begin{equation}
\label{eq:quantilehere}\mathbf{P}_{\pi^{\otimes k}} \bigl(M_{\{
i,j\} }\leq
\eps\mathsf{m}(Q) \bigr)= O (\eps ).
\end{equation}
Now note that Propositions~\ref{prop:correltransitive} (in the reversible/transitive
case) and~\ref{prop:correlgeneral} (in the general case) imply that each
term in the sum defining $\xi$ is $O (\mathsf{err}(Q)^2
)$. We deduce
\[
\xi\leq\frac{O (\eps^2 ) \bigl({\ell\atop2}\bigr)}{\ell\eps
}\leq O (\ell\eps ) = O \bigl(k^2
\mathsf{err}(Q) \bigr).
\]
\upqed\end{pf}

\section{Coalescing random walks: Basics}\label{sec:CRW}

In this section we formally define the coalesing random walks process.
We then show that if the initial number of particles is not large, mean
field behavior follows from the exponential approximation of meeting times.
\subsection{Definitions}

Fix a Markov chain $Q$ on a finite state space $\bV$. Given a number
$k\in[|\bV|]\setminus\{1\}$ and an initial state $x^{(k)}\in\bV^k$,
consider a realization of $Q^{(k)}$
\[
\bigl(X^{(k)} \bigr)_{t\geq0}\equiv \bigl(X_t(1),
\ldots,X_t(k) \bigr)_{t\geq0}.
\]
We build the coalescing random walks process from $X^{(k)}$ by defining
the trajectories of the $k$ walkers one by one. We first set
\[
\overline{X}_t(1) = X_t(1),\qquad t\geq0.
\]
Given $j\in[k]\setminus\{1\}$, assume that $\overline{X}_t(i)$ has
been defined for all $1\leq i<j$ and $t\geq0$. We let $T_j$ be the
first time $t\geq0$ at which $X_t(j)=\overline{X}_t(I_j)$ for some
$1\leq I_j<i$, and then set
\[
\overline{X}_t(j)\equiv\cases{ %
X_t(j),&\quad $t<T_j;$
\vspace*{2pt}\cr
\overline{X}_t(I_j), &\quad $t\geq T_j.$}
\]
Intuitively, this says that as soon as $j$ encounters a walker with
lower index, it starts moving along with it. The process
\[
\bigl(\overline{X}^{(k)}_t \bigr)_{t\geq0}\equiv
\bigl(\overline{X}_t(j) \bigr)_{t\geq0}
\]
is what we call the \emph{coalescing random walks process} based on $Q$,
with initial state~$x^{(k)}$.

%
\begin{remark} For any $j\geq3$, there might be more than one index
$i<j$ such that $\overline{X}_{T_j}(i)=X_{T_j}(j)$. However, it is easy
to see that all such $i$ will have the same trajectory after time $T_j$
because they must have met by that time. This implies that there is no
ambiguity in the definition of $\overline{X}_t(j)$ for any $j$.
\end{remark}

We also define
\[
\Ctime_i\equiv\inf \bigl\{t\geq0 \dvtx \bigl| \bigl\{\overline{X}_t(j)
\dvtx j\in[k] \bigr\} \bigr|\leq i \bigr\}
\]
and $\Ctime\equiv\Ctime_1$. The fact that we are working in continuous
time implies the following:

%
\begin{proposition}[(Proof omitted)]Assume that the initial state
$x^{(k)}=(x(1),x(2),\ldots,x(k))$ is such that $x(i)\neq x(j)$ for all
$1\leq i<j\leq k$. Then $\Ctime_k=0<\Ctime_{k-1}<\Ctime_{k-2}<\cdots
<\Ctime_1$ almost surely.
\end{proposition}

It is sometimes useful to view the coalescing random walks process as a
process with \emph{killings}. Define a random $2^{[k]}$-valued process
$(A_t)_{t\geq0}$ as follows:
\begin{itemize}
\item$1\in A_t$ for all $t$;
\item proceeding recursively, for each $j\in[k]\setminus\{1\}$, we
have $j\in A_t$ if and only if $\tau_j>t$, where $\tau_j$ is the first
time $t$ at which $X_t(i)=X_t(j)$ for some $i<j$ with $i\in A_t$.
\end{itemize}

Intuitively, $A_t$ is the set of all walkers that are ``alive'' at time
$t\geq0$, and a walker dies at the first time it meets an alive walker
with smaller index. One may check that coalescing random walks is
equivalent to the killed process in the following sense.

%
\begin{proposition}[(Proof omitted)] We have $\tau_j=T_j$ for all $j\in
[k]\setminus\{1\}$. Moreover, for all $t\geq0$, we have
\[
\bigl\{X_t(j) \dvtx j\in A_t \bigr\} = \bigl\{
\overline{X}_t(j) \dvtx j\in[k] \bigr\}.
\]
Finally, for all $i\in[k-1],$
\[
\Ctime_i =\inf\bigl\{t\geq0 \dvtx |A_t|\leq i\bigr\}.
\]
\end{proposition}

Recall that $M_{i,j}$ is the meeting time between walkers $i$ and $j$;
cf. \eqnref{defMij}. We have the following simple proposition:
%
%
\begin{proposition}[(Proof omitted)]\label{prop:incrementofC}Assume that
the initial state
\[
x^{(k)}= \bigl(x(1),x(2),\ldots,x(k) \bigr)
\]
is such that $x(i)\neq x(j)$ for all $1\leq i<j\leq k$. Then for each
$1\leq p\leq k-1$,
\[
\Ctime_p - \Ctime_{p+1} = \min_{\{i,j\}\subset A_{\Ctime
_{p+1}}}M_{i,j}
\circ\Theta_{\Ctime_{p+1}}.
\]
Moreover, each time $\Ctime_p$ equals $M_{i,j}$ for some $\{i,j\}\in
\bigl({[k]\atop2}\bigr)$.
\end{proposition}

\subsection{Mean-field behavior for moderately large $k$}

We now prove a mean-field-like result for an initial number of
particles $k$ that is not too large, assuming that meeting times of up
to $k$ walkers satisfy our exponential approximation property.

%
\begin{lemma}\label{lem:inductiveargument}Assume that $Q$, $\mathsf
{err}(Q)$ and
$k$ satisfy the assumptions of \lemref{firstmeetingtimemultiple}. Let
$x^{(k)}\in\bV^k$. Then for all $p\in[k-1],$
\[
d_W \Biggl(\operatorname{Law}_{x^{(k)}} \biggl(
\frac{\Ctime
_p}{\mathsf{m}(Q)} \biggr),\operatorname{Law} \Biggl(\sum
_{i=p+1}^{k}\mathsf{Z}_{i} \Biggr) \Biggr)=
\frac{O (k^2\mathsf{err}(Q) ) +12 \eta(x^{(k)})}{p},
\]
where
\begin{eqnarray*}
\eta \bigl(x^{(k)} \bigr) &=& \mathbf{P}_{x^{(k)}}
\bigl(M^{(k)}\leq t_{\mathrm{mix}}^Q \bigl(
\mathsf{err}(Q)^2 \bigr) \bigr)
\\
&&{}+\mathbf{P}_{x^{(k)}} \pmatrix{
\exists\{i,j\},
\{\ell,r\} \in \pmatrix{[k]\cr 2}  \dvtx \vspace*{2pt}\cr
\{ \ell,r\}\neq\{i,j\}\mbox{ but } M_{i,j}\circ \Theta_{M_{\ell,r}}
\leq t_{\mathrm{mix}}^Q \bigl(\mathsf {err}(Q)^2 \bigr)
},
\end{eqnarray*}
and the $\mathsf{Z}_{i}$ are the random variables described in \eqnref{defZi}.
\end{lemma}

\begin{pf}Write $x^{(k)}=(x(1),\ldots,x(k))$. We will prove the
similar bound
%
%
\begin{eqnarray}
\label{eq:goaldistinct}&&\mbox{``}\forall1\leq i<j\leq k \dvtx x(i)\neq x(j)\mbox{''}
\nonumber\\
&&\qquad\Rightarrow d_W \Biggl(\operatorname{Law}_{x^{(k)}} \biggl(
\frac{\Ctime
_p}{\mathsf{m}(Q)} \biggr),\operatorname{Law} \Biggl(\sum
_{i=p+1}^{k}\mathsf{Z}_{i} \Biggr) \Biggr)\\
&&\qquad=
\frac{O (k^2\mathsf{err}(Q) )
+4 \eta(x^{(k)})}{p}.\nonumber
\end{eqnarray}
To see how this implies the general result, consider some $x^{(k)}$
such that some of its coordinates are equal, so that in particular
$\eta
(x^{(k)})\geq1$. One still has the trivial bound
\[
d_W \Biggl(\operatorname{Law}_{x^{(k)}} \biggl(
\frac{\Ctime
_p}{\mathsf{m}(Q)} \biggr),\operatorname{Law} \Biggl(\sum
_{i=p+1}^{k}\mathsf{Z}_{i} \Biggr) \Biggr)
\leq \mathbf{E}_{x^{(k)}} \biggl[\frac{\Ctime _p}{\mathsf{m}(Q)} \biggr] + \mathbf{E}
\Biggl[\sum_{i=p+1}^{k}\mathsf{Z}_{i}
\Biggr].
\]
The second term on the RHS is $\leq2/p$. For the first term, let $j$
be the number of distinct coordinates of $x$ and
\[
y^{(j)}= \bigl(y(1),\ldots,y(j) \bigr)\in\bV^j
\]
have distinct coordinates with
\[
\bigl\{y(1),\ldots,y(j) \bigr\} = \bigl\{x(1),\ldots,x(k) \bigr\}.
\]
Then clearly,
\[
\mathbf{E}_{x^{(k)}} \biggl[\frac{\Ctime_p}{\mathsf{m}(Q)} \biggr]=
\mathbf{E}_{y^{(j)}} \biggl[\frac{\Ctime _p}{\mathsf{m}(Q)} \biggr].
\]
If $p\geq j$, the RHS is $0$. If not, it can be upper bounded using the
bound in~\eqnref{goaldistinct},
\begin{eqnarray*}
\mathbf{E}_{y^{(j)}} \biggl[\frac{\Ctime_p}{\mathsf{m}(Q)} \biggr]&\leq&\mathbf{E}
\Biggl[\sum_{i=p+1}^{k}\mathsf{Z}_{i}
\Biggr] + d_W \Biggl( \operatorname{Law}_{y^{(j)}} \biggl(
\frac{\Ctime _p}{\mathsf{m}(Q)
} \biggr),\operatorname{Law} \Biggl(\sum
_{i=p+1}^{j}\mathsf {Z}_{i} \Biggr) \Biggr)
\\
&\leq&\frac{2+ 4\eta(y^{(j)})
+ O (k^2\mathsf{err}(Q) )}{p}.
\end{eqnarray*}
Since $\eta(x^{(k)})\geq1\geq\eta(y^{(j)})/2$ in this case, we obtain
\[
d_W \Biggl(\operatorname{Law}_{x^{(k)}} \biggl(
\frac{\Ctime
_p}{\mathsf{m}(Q)} \biggr),\operatorname{Law} \Biggl(\sum
_{i=p+1}^{k}\mathsf{Z}_{i} \Biggr) \Biggr)
\leq \frac{12 \eta(x^{(k)}) + O (k^2\mathsf{err}(Q) )}{p}
\]
for such $x^{(k)}$ with repetitions, which gives the lemma in general.

We prove \eqnref{goaldistinct} by reverse induction on $p$. The case
$p=k-1$ is trivial: $\Ctime_{k-1}$ is simply $M^{(k)}$, and $\eta
(x^{(k)})$ is an upper bound for $r_{\delta_{x^{(k)}}}$, so we may
apply \lemref{firstmeetingtimemultiple} to deduce the desired bound.

For the inductive step, consider $p_0<k-1$, and assume the result is
true for all $p_0<p\leq k-1$. We will use the easily proven fact that
$\Ctime_{p_0+1}$ is a stopping time for the process
$(X^{(k)}_t)_{t\geq
0}$ process. Consider the corresponding $\sigma$-field $\sF_{\Ctime
_{p_0+1}}$. We will apply \lemref{conditionalwasserstein} with
\begin{eqnarray*}
Z_1 &=& \sum_{i=p_0+2}^k
\mathsf{Z}_{i},
\\
Z_2&=& \mathsf{Z}_{p_0+1},
\\
W_1&=& \frac{\Ctime_{p_0+1}}{\mathsf{m}(Q)},
\\
W_2&=&\frac{\Ctime_{p_0}-\Ctime_{p_0+1}}{\mathsf{m}(Q)} = \frac
{\min_{\{i,j\}\subset A_{\Ctime_{p_0+1}}}M_{i,j}\circ\Theta
_{\Ctime_{p_0+1}}}{\mathsf{m}(Q)},
\\
\sG&=& \sF_{\Ctime_{p_0+1}}.
\end{eqnarray*}
(We used \propref{incrementofC} to obtain the second expression for
$W_2$ above.) Applying \lemref{conditionalwasserstein} in conjunction
with the induction hypothesis gives
%
%
\begin{eqnarray}
\label{eq:decompdW}&& d_W \Biggl(\operatorname{Law}_{x^{(k)}}
\biggl(\frac {\Ctime _{p_0}}{\mathsf{m}(Q)} \biggr),\sum_{i=p_0+1}^k
\mathsf{Z}_{i} \Biggr)\nonumber\\
&&\qquad\leq\frac{O (k^2\mathsf{err}(Q)  ) + 4\eta
(x^{(k)})}{p_0+1}
\nonumber
\\[-8pt]
\\[-8pt]
\nonumber
&&\qquad\quad{}+ \mathbf{E}_{x^{(k)}} \biggl[d_W \biggl(\operatorname
{Law}_{x^{(k)}} \biggl(\frac{\min_{\{i,j\}\subset A_{\Ctime
_{p_0+1}}}M_{i,j}\circ\Theta_{\Ctime_{p_0+1}}}{\mathsf
{m}(Q)}\Big\vert\sF_{C_{p_0+1}}
\biggr),\\
&&\hspace*{285pt}\mathsf{Z}_{p_0+1} \biggr) \biggr].\nonumber
\end{eqnarray}
Note that $A_{\Ctime_{p_0+1}}$ is $\sF_{\Ctime_{p_0+1}}$-measurable.
The strong Markov property for $Q^{(k)}$ implies that
\[
\operatorname{Law}_{x^{(k)}} \biggl(\frac{\min_{\{i,j\}\subset
A_{\Ctime _{p_0+1}}}M_{i,j}\circ\Theta_{\Ctime_{p_0+1}}}{\mathsf
{m}(Q)}\Big\vert
\sF_{C_{p_0+1}} \biggr)
\]
is the same as
\[
\operatorname{Law}_{X^{(k)}_{\Ctime_{p_0+1}}} \biggl(\frac{\min_{\{
i,j\}\subset A_{\Ctime _{p_0+1}}}M_{i,j}}{\mathsf{m}(Q)} \biggr).
\]
Now define $Y^{(p_0+1)}$ as the vectors whose coordinates are the
$p_0+1$ distinct points $X_{\Ctime_{p_0+1}}(i)$ with $i\in A_{p_0+1}$
(the order of the coordinates does not matter). Clearly,
%
%
\begin{equation}
\label{eq:choiceofYk}\operatorname{Law}_{X^{(k)}_{\Ctime
_{p_0+1}}} \biggl(
\frac{\min_{\{i,j\}\subset A_{\Ctime
_{p_0+1}}}M_{i,j}}{\mathsf{m}(Q) } \biggr)=\operatorname {Law}_{Y^{(p_0+1)}} \biggl(
\frac{M^{(p_0+1)}}{\mathsf{m}(Q)} \biggr).
\end{equation}
By \lemref{firstmeetingtimemultiple}, this last law is approximately
exponential,
\[
\Expdist \biggl(\frac{1}{\bigl({p_0+1 \atop2}\bigr)},O \bigl(k^2\mathsf {err}(Q)
\bigr)+2 r_{\delta
_{Y^{(p_0+1)}}},O \bigl(k^2 \mathsf{err}(Q) \bigr) \biggr),
\]
and \lemref{approxwasserstein} gives
\[
d_W \biggl(\operatorname{Law}_{Y^{(p_0+1)}} \biggl(
\frac
{M^{(p_0+1)}}{\mathsf{m}(Q)} \biggr),\mathsf{Z}_{p_0+1} \biggr)\leq
\frac{O (k^2\mathsf{err}(Q) ) + 4r_{\delta
_{Y^{(p_0+1)}}}}{p_0(p_0+1)}.
\]
Using the definition of $r_{\delta_{Y^{(p_0+1)}}}$, we obtain from
\eqnref{decompdW} the following inequality:
%
%
\begin{eqnarray}
\label{eq:decompdWback}&& d_W \Biggl(\operatorname{Law}_{x^{(k)}}
\biggl(\frac {\Ctime_{p_0}}{\mathsf{m}(Q)} \biggr),\sum_{i=p_0+1}^k
\mathsf{Z}_{i} \Biggr)\nonumber\\
&&\qquad\leq\frac
{O (k^2\mathsf{err}(Q) ) + 4\eta(x^{(k)})}{p_0+1}
\\
&&\qquad\quad{}+\frac{O (k^2\mathsf{err}(Q) ) +
4\mathbf{E}_{x^{(k)}} [\mathbf{P}_{{Y^{(p_0+1)}}}
(M^{(p_0+1)}\leq t_{\mathrm{mix}}^Q(\mathsf{err}(Q)^2) )
]}{p_0(p_0+1)}.\nonumber
\end{eqnarray}
To finish, we need to show that the expected value on the RHS is $\leq
\eta(x^{(k)})$. For this we recall \eqnref{choiceofYk} to note that
\begin{eqnarray*}
& & \mathbf{P}_{Y^{(p_0+1)}} \bigl(M^{(p_0+1)}\leq t_{\mathrm
{mix}}^Q
\bigl( \mathsf{err}(Q)^2 \bigr) \bigr)
\\
&&\qquad=\mathbf{P}_{X^{(k)}_{\Ctime_{p_0+1}}} \Bigl(\min_{\{i,j\}\subset
A_{\Ctime_{p_0+1}}}M_{i,j}
\leq t_{\mathrm{mix}}^Q \bigl(\mathsf {err}(Q)^2 \bigr)
\Bigr)
\\
&&\qquad= \mathbf{P}_{x^{(k)}} \bigl(\Ctime_{p_0}-
\Ctime_{p_0+1}\leq t_{\mathrm{mix}}^Q \bigl(
\mathsf{err}(Q)^2 \bigr)\mid\sF_{\Ctime
_{p_0+1}} \bigr),
\end{eqnarray*}
where the last line uses \propref{incrementofC} and the strong Markov
property. Averaging shows that the expectation on the RHS of \eqnref
{decompdWback} is
\[
\mathbf{P}_{x^{(k)}} \bigl(\Ctime_{p_0}-\Ctime_{p_0+1}
\leq t_{\mathrm{mix}}^Q \bigl(\mathsf{err}(Q)^2 \bigr)
\bigr),
\]
and \propref{incrementofC} implies that this is at most
\[
\mathbf{P}_{x^{(k)}} \biggl(\bigcup_{\{i,j\}\neq\{\ell,r\}}
\bigl\{ M_{i,j}\circ\Theta_{M_{\ell,r}}\leq t_{\mathrm{mix}}^Q
\bigl( \mathsf{err}(Q)^2 \bigr) \bigr\} \biggr).
\]
Since the RHS is $\leq\eta(x^{(k)})$, we are done.\
\end{pf}

\section{Proofs of the main theorems}\label{sec:proofsmain}

\subsection{The full coalescence time in the transitive case}\label{sec:thm:transitive_proof}

In this section we prove \thmref{transitive}.

\begin{pf*}{Proof of \thmref{transitive}} Recall that $\Ctime=\Ctime
_1$ by definition. \lemref{inductiveargument} gives the following bound
for any $k\leq\sqrt{1/4\mathsf{err}(Q)}\wedge|\bV|$ and
$x^{(k)}\in\bV^k$:
\[
d_W \Biggl(\operatorname{Law}_{x^{(k)}} \biggl(
\frac{\Ctime}{\mathsf
{m}(Q)} \biggr),\sum_{i=2}^{k}
\mathsf{Z}_{i} \Biggr)\leq12 \eta \bigl(x^{(k)} \bigr) + O
\bigl(k^2\mathsf{err}(Q) \bigr).
\]
Notice that
\[
d_W \Biggl(\sum_{i=2}^{k}
\mathsf{Z}_{i},\sum_{i=2}^{+\infty}
\mathsf{Z}_{i} \Biggr)\leq \mathbf{E} \biggl[\sum
_{j\geq k+1}\mathsf{Z}_{j} \biggr] = \frac{2}{k+1},
\]
hence
\[
d_W \Biggl(\operatorname{Law}_{x^{(k)}} \biggl(
\frac{\Ctime
_1}{\mathsf{m}(Q)} \biggr),\sum_{i=2}^{+\infty}
\mathsf{Z}_{i} \Biggr) =12 \eta \bigl(x^{(k)} \bigr) +O
\biggl(k^2\mathsf{err}(Q)+\frac{1}{k} \biggr).
\]
Convexity of $d_W$ implies

%
\begin{proposition}\label{prop:transitivemoderatekl}Under the
assumptions of \thmref{transitive}, the following holds for $\mathsf
{err}(Q)\leq
1/4$, $1\leq k\leq\sqrt{1/4\mathsf{err}(Q)}\wedge|\bV|$ and
$\lambda
^{(k)}\in
M_1(\bV^k)$:
%
%
\begin{eqnarray}
\label{eq:moderatekI}d_W \Biggl(\operatorname{Law}_{\lambda
^{(k)}}
\biggl(\frac {\Ctime_1}{\mathsf{m}(Q)} \biggr),\sum_{i=2}^{+\infty}
\mathsf{Z}_{i} \Biggr) &\leq& 12 \int\eta \bigl(x^{(k)} \bigr)
\,d\lambda^{(k)} \bigl(x^{(k)} \bigr)
\nonumber
\\[-8pt]
\\[-8pt]
\nonumber
& &{} + O \biggl(k^2\mathsf{err}(Q)+ \frac{1}{k}
\biggr).
\end{eqnarray}
\end{proposition}

Notice that our control of $\Ctime_1$ gets worse as $k$ increases, and
we cannot use the above bound to approximate the law of $\Ctime_1$
started with one particle at each vertex of $\bV$. What we use instead
is a truncation argument combined with the Sandwich lemma for $d_W$
(\lemref{sandwich} above). For this we need to find two random variables
\[
\Ctime_-\preceq_d\Ctime_1\preceq_d\Ctime_+
\]
such that both $\Ctime_-/\mathsf{m}(Q)$ and $\Ctime_+/\mathsf
{m}(Q)$ are close to
$\sum_{i=2}^{+\infty}\mathsf{Z}_{i}$. More specifically, we will
show that
\[
\label{eq:plusminus}d_W \biggl(\frac{\Ctime_\pm}{\mathsf{m}(Q)
},\sum
_{i\geq2}\mathsf{Z}_{i} \biggr) = O
\biggl(k^2 \mathsf{err}(Q)+ k^4\mathsf{err}(Q)^2+
\frac {1}{k+1} + \rho(Q)\ln\bigl(1/\rho(Q)\bigr) \biggr).
\]
Before we continue, let us show how this last bound implies our result.
The Sandwich \lemref{sandwich} gives
\[
d_W \biggl(\frac{\Ctime_1}{\mathsf{m}(Q)},\sum_{i\geq2}
\mathsf{Z}_{i} \biggr) = O \biggl(k^2 \mathsf{err}(Q)+
k^4\mathsf {err}(Q)^2+ \frac{1}{k} + \rho(Q)\ln
\frac{1}{\rho(Q)} \biggr).
\]
Since $\rho(Q)\ln(1/\rho(Q))=O (\mathsf{err}(Q) )$, we
may choose $k=(\mathsf{err}(Q)
)^{-1/3}$ [which works for $\mathsf{err}(Q)$ sufficiently small] to obtain
\[
d_W \biggl(\frac{\Ctime_1}{\mathsf{m}(Q)},\sum_{i\geq2}
\mathsf{Z}_{i} \biggr) = O \bigl(\mathsf{err}(Q)^{1/3}
\bigr),
\]
and this is precisely the bound we seek because
\[
\mathsf{err}(Q)=O \bigl(\sqrt{\rho(Q)\ln\bigl(1/\rho(Q)\bigr)} \bigr).
\]
We now construct $\Ctime_-,\Ctime_+$ and prove that they have the
required properties.

\textit{Construction of $\Ctime_-$}: pick $x(1),\ldots,x(k)\in\bV$ from
distribution $\pi$, independently and with replacement. Let $\Ctime_-$
denote the full coalescence time for $k$ walkers started from these
positions. This might be degenerate: there might be more than one
walker starting from some element of $\bV$, but this only means those
particles will coalesce instantly.

Clearly, $\Ctime_-\preceq_d\Ctime_1$. Moreover,
\[
\operatorname{Law} \biggl(\frac{\Ctime_-}{\mathsf{m}(Q)} \biggr)=\operatorname{Law}_{\pi^{\otimes k}}
\biggl(\frac{\Ctime
_1}{\mathsf{m}(Q)} \biggr).
\]
Therefore by \propref{transitivemoderatekl},
\begin{eqnarray*}
d_W \Biggl(\operatorname{Law} \biggl(\frac{\Ctime_-}{\mathsf
{m}(Q)} \biggr),
\sum_{i=2}^{+\infty}\mathsf{Z}_{i}
\Biggr)&=& d_W \Biggl( \operatorname{Law}_{\pi^{\otimes k}} \biggl(
\frac {\Ctime
_1}{\mathsf{m}(Q)} \biggr),\sum_{i=2}^{+\infty}
\mathsf{Z}_{i} \Biggr)
\\
&= & O \biggl(\int\eta \bigl(x^{(k)} \bigr) \,d\pi^{\otimes k}+
k^2 \mathsf{err}(Q)+ \frac{1}{k} \biggr).
\end{eqnarray*}
Notice that the integral on the RHS is at most
%
%
\begin{eqnarray}
\label{eq:integralGRU}\int\eta \bigl(x^{(k)} \bigr) \,d
\pi^{\otimes k}& \leq& \sum_{\{i,j\}\in({[k] \atop 2})}
\mathbf{P}_{\pi^{\otimes
k}} \bigl(M_{i,j}\leq t_{\mathrm{mix}}^Q
\bigl(\mathsf{err}(Q)^2 \bigr) \bigr)
\nonumber\\
& &{} + \sum_{\stackrel{\{i,j\},\{\ell,r\}\in({[k] \atop 2}) \dvtx }{\{
i,j\}
\neq\{\ell,r\}}}\mathbf{P}_{\pi^{\otimes k}}
\pmatrix{M_{i,j}\circ\Theta_{M_{\ell,r}}
\vspace*{2pt}\cr
\leq t_{\mathrm
{mix}}^Q \bigl(\mathsf{err}(Q)^2 \bigr)
}
\\
&=&O \bigl(k^4\mathsf{err}(Q)^2 \bigr)\nonumber
\end{eqnarray}
as can be deduced from the proofs of Propositions \ref
{prop:correltransitive} and~\ref{prop:goodbound}. We conclude that
%
%
\begin{eqnarray}
\label{eq:wassersteinlower}&& d_W \biggl(\operatorname{Law} \biggl(
\frac {\Ctime _-}{\mathsf{m}(Q)} \biggr), \operatorname{Law} \biggl(\sum
_{i\geq2}\mathsf{Z}_{i} \biggr) \biggr)
\nonumber
\\[-8pt]
\\[-8pt]
\nonumber
&&\qquad=O \biggl(k^4\mathsf{err}(Q)^2 +
k^2 \mathsf{err}(Q)+ \frac
{1}{k} \biggr).
\end{eqnarray}

\textit{Construction of $\Ctime_+$}: we will use the following simple
stochastic domination result, which we describe in the language of the
process with killings. Let $\tau\leq\sigma$ be stopping times for the
$X^{(k)}$ process. If all killings are suppressed between time~$\tau$
and $\sigma$, the resulting full coalescence time $\Ctime_+$
stochastically dominates $\Ctime_1$. We will use this result, whose
proof we omit, with the following choice of $\tau$ and $\sigma$:
\[
\tau= \Ctime_k\quad\mbox{and}\quad\sigma= \Ctime_k +
t_{\mathrm{mix}}^Q \bigl(\mathsf{err}(Q)^2 \bigr).
\]
\lemref{sumW} implies
\[
d_W \biggl(\frac{\Ctime_+}{\mathsf{m}(Q)},\frac{\Ctime_1\circ
\Theta
_\sigma}{\mathsf{m}(Q)
} \biggr)\leq
\frac{\mathbf{E} [\sigma ]}{\mathsf{m}(Q)} = \frac
{\mathbf{E} [\Ctime _k ]}{\mathsf{m}(Q)
}+\frac{t_{\mathrm{mix}}^Q(\mathsf{err}(Q)^2)}{\mathsf{m}(Q)}.
\]
Since $Q$ is transitive, $\mathsf{m}(Q)$ can be bounded from below in
terms of
the maximal hitting time in $Q$~\cite{AldousFill_RWBook}, Chapter 14.
Theorem 1.2 in~\cite{Oliveira_TAMS} implies
\[
\mathbf{E} [\Ctime_k ]\leq\frac{C \mathsf{m}(Q)}{k} + C
t_{\mathrm{mix}}^Q
\]
for some universal $C>0$. Recalling the definition of $\rho(Q)$ in
\eqnref{defratio}, we obtain
\[
\frac{\mathbf{E} [\Ctime_k ]}{\mathsf{m}(Q)} = O \biggl(\frac{1}{k} + \rho(Q) \biggr).
\]
Moreover, we also have
\[
t_{\mathrm{mix}}^Q \bigl(\mathsf{err}(Q)^2 \bigr)=O
\bigl(\ln \bigl(1/\mathsf{err}(Q)\bigr) t_{\mathrm{mix}}^Q \bigr)= O
\bigl(t_{\mathrm
{mix}}^Q\ln\bigl(1/\rho(Q)\bigr) \bigr),
\]
hence
\[
d_W \biggl(\frac{\Ctime_+}{\mathsf{m}(Q)},\frac{\Ctime_1\circ
\Theta
_\sigma}{\mathsf{m}(Q)
} \biggr)= O \biggl(
\frac{1}{k} + \rho(Q)\ln\bigl(1/\rho(Q)\bigr) \biggr).
\]
This shows
\[
d_W \Biggl(\frac{\Ctime_+}{\mathsf{m}(Q)},\sum_{i=2}^k
\mathsf{Z}_{i} \Biggr)= O \biggl(\frac {1}{k} + \rho(Q)\ln
\bigl(1/\rho (Q)\bigr) \biggr) + d_W \Biggl(\frac{\Ctime_1\circ
\Theta
_\sigma}{\mathsf{m}(Q)},\sum
_{i=2}^k \mathsf{Z}_{i}
\Biggr).
\]
Now consider the time $\Ctime_1\circ\Theta_\sigma$. Since all killings
were suppressed between times $\tau=\Ctime_k$ and $\sigma=\Ctime
_k+t_{\mathrm{mix}}
^Q(\mathsf{err}(Q)^2)$, there are $k$ alive particles at time $\sigma
_-$. Letting
$\lambda^{(k)}$ denote their law, we have
\[
\operatorname{Law} \biggl(\frac{\Ctime_1\circ\Theta_\sigma
}{\mathsf{m}(Q)} \biggr) = \operatorname{Law}_{\lambda^{(k)}}
\biggl(\frac{\Ctime_1}{\mathsf{m}(Q)} \biggr),
\]
and \propref{transitivemoderatekl} implies
\[
d_W \Biggl(\frac{\Ctime_1\circ\Theta_\sigma}{\mathsf{m}(Q)},\sum_{i=2}^k
\mathsf{Z}_{i} \Biggr)=O \biggl(\int\eta \bigl(x^{(k)} \bigr)
\,d \lambda^{(k)} \bigl(x^{(k)} \bigr) + k^2
\mathsf{err}(Q)+ \frac
{1}{k} \biggr).
\]
Now observe that
\[
t_{\mathrm{mix}}^Q \bigl(\mathsf{err}(Q)^2 \bigr)\geq
t_{\mathrm
{mix}}^{Q^{(k)}} \bigl(k \mathsf{err}(Q)^2 \bigr)
\qquad\mbox{(cf. \lemref{productchains})},
\]
hence the law of the $k$ particles at time $\Ctime_k+t_{\mathrm
{mix}}^Q(\mathsf{err}(Q)^2)$
is $k\mathsf{err}(Q)^2$-close to stationary, irrespective of their
states at time
$\Ctime_k$. We deduce that $\lambda^{(k)}$ is $k\mathsf
{err}(Q)^2$-close to
stationary, and
\[
d_W \Biggl(\frac{\Ctime_1\circ\Theta_\sigma}{\mathsf{m}(Q)},\sum_{i=2}^k
\mathsf{Z}_{i} \Biggr)= O \biggl(\int\eta \bigl(x^{(k)}
\bigr) \,d\pi ^{\otimes k} + k^2\mathsf{err}(Q)+ k
\mathsf{err}(Q)^2+ \frac
{1}{k} \biggr).
\]
The integral on the RHS was estimated in \eqnref{integralGRU}, and we deduce
\[
d_W \Biggl(\frac{\Ctime_1\circ\Theta_\sigma}{\mathsf{m}(Q)},\sum_{i=2}^k
\mathsf{Z}_{i} \Biggr)=O \biggl(k^2 \mathsf{err}(Q)+
k^4\mathsf {err}(Q)^2+ \frac{1}{k} \biggr),
\]
and we deduce
\[
d_W \Biggl(\frac{\Ctime_+}{\mathsf{m}(Q)},\sum_{i=2}^k
\mathsf{Z}_{i} \Biggr) = O \biggl(k^2 \mathsf{err}(Q)+
k^4\mathsf {err}(Q)^2+ \frac{1}{k} + \rho(Q)\ln
\bigl(1/\rho(Q)\bigr) \biggr).\qquad
\]
\upqed\end{pf*}

\subsection{The general setting}\label{sec:thm_general}

We now come to the proof of \thmref{general}.
\begin{pf*}{Proof of \thmref{general}} The proof is essentially the
same as in the reversible/transitive case, but with the definition of
$\mathsf{err}(Q)$ given in \eqnref{deferrtwo}. In particular, we can
still use the
same definition of $\Ctime_-$ used in that proof to obtain
%
%
\begin{equation}
\label{eq:Ctime-general}d_W \Biggl(\frac{\Ctime
_-}{\mathsf{m}(Q)
},\sum
_{i=2}^{+\infty}\mathsf{Z}_{i} \Biggr) = O
\biggl(k^2 \mathsf {err}(Q)+ k^4\mathsf{err}(Q)^2+
\frac{1}{k} \biggr).
\end{equation}

We will need a different strategy in the analysis of $\Ctime_+$, where
we need to bound $\mathbf{E} [\Ctime_k ]$ by different
means. Note that
$\Ctime
_k\geq t$ if and only if there exist distinct $y(1),\ldots,y(k)\in\bV$
such that there is \emph{no coalescence} among the walkers started from
these vertices. The probability of this ``no coalescence event'' for a
given choice of $y(i)$'s is $\mathbf{P}_{y^{(k)}} (M^{(k)}\geq
t )$ for
$y^{(k)}=(y(1),\ldots,y(k))$. Therefore,
\[
\mathbf{P} (\Ctime_k\geq t )\leq \biggl(\sum
_{y^{(k)}\in
\bV^k} \mathbf{P}_{y^{(k)}} \bigl(M^{(k)}\geq t
\bigr) \biggr)\wedge1.
\]
By \lemref{firstmeetingtimemultiple}, each term in the RHS satisfies
\[
\mathbf{P}_{y^{(k)}} \bigl(M^{(k)}\geq t \bigr)\leq C
e^{-{t
({k\atop2})}/{((1+O (k^2\mathsf{err}(Q) )) \mathsf{m}(Q))}}
\]
for some universal $C>0$. Since there are $\leq|\bV|^k$ terms in the
sum, we have
\[
\mathbf{P} (\Ctime_k\geq t )\leq \bigl(C |\bV|^k
e^{-{t ({k \atop
2})}/{((1+O (k^2\mathsf{err}(Q) )) \mathsf{m}(Q))}} \bigr)\wedge1.
\]
Integrating the RHS gives
\[
\frac{\mathbf{E} [\Ctime_k ]}{\mathsf{m}(Q)}\leq C \frac
{\ln|\bV|}{k}
\]
for a potentially different, but still universal $C$. Going through the
previous proof, we see that this gives
%
%
\begin{eqnarray}
\label{eq:Ctimeplusgeneral}&&d_W \Biggl(\frac{\Ctime
_+}{\mathsf{m}(Q)},\sum
_{i=2}^k\mathsf{Z}_{i}
\Biggr)
\nonumber
\\[-8pt]
\\[-8pt]
\nonumber
&&\qquad=O \biggl(k^2 \mathsf{err}(Q)+ k^4
\mathsf{err}(Q)^2+ \frac{\ln|\bV|}{k} + \frac{t_{\mathrm
{mix}}^Q(\mathsf{err}(Q)^2)}{\mathsf{m}(Q)} \biggr).
\end{eqnarray}
To continue, we bound the term containing $t_{\mathrm{mix}}^Q(\mathsf
{err}(Q)^2)$ in terms of
$\mathsf{err}(Q)$ [this was easier before because of the different
definition of
$\mathsf{err}(Q)$]. Recall from \propref{goodboundgeneral} that
\[
\mathbf{P}_{\pi^{\otimes2}} \bigl(M\leq t_{\mathrm{mix}}^Q \bigl(
\mathsf{err}(Q)^2 \bigr) \bigr)\leq\mathsf{err}(Q)^2.
\]
Therefore, for all $j\in\N$,
\begin{eqnarray*}
&& \mathbf{P}_{\pi^{\otimes2}} \bigl(M\leq j
t_{\mathrm{mix}}^Q \bigl(\mathsf{err}(Q)^2 \bigr) \bigr)
\\
&&\qquad\leq \sum_{i=1}^j\mathbf{P}_{\pi^{\otimes2}}
\bigl(M\circ\Theta _{(i-1)t_{\mathrm{mix}} ^Q(\mathsf{err}(Q) ^2)} t_{\mathrm
{mix}}^Q \bigl(
\mathsf{err}(Q)^2 \bigr) \bigr)
\\
&&\qquad\leq j\mathsf{err}(Q)^2.
\end{eqnarray*}
On the other hand, taking
\[
j= \biggl\lceil\frac{2\mathbf{E}_{\pi^{\otimes2}} [M
]}{t_{\mathrm{mix}}^Q(\mathsf{err}(Q)
^2)} \biggr\rceil,
\]
we obtain
\[
\mathbf{P}_{\pi^{\otimes2}} \bigl(M\leq j t_{\mathrm{mix}}^Q \bigl(
\mathsf{err}(Q)^2 \bigr) \bigr)\geq1- \frac{\mathbf{E}_{\pi
^{\otimes2}} [M ]}{j t_{\mathrm{mix}}^Q(\mathsf
{err}(Q)^2)}\geq
\frac{1}{2}.
\]
Combining these two inequalities gives
\[
\frac{t_{\mathrm{mix}}^Q(\mathsf{err}(Q)^2)}{\mathbf{E}_{\pi
^{\otimes2}} [M ]} = O \bigl(\mathsf{err}(Q)^2 \bigr).
\]
This implies that the term containing $t_{\mathrm{mix}}^Q(\mathsf
{err}(Q)^2)$ on the RHS of
\eqnref{Ctimeplusgeneral} can be neglected. Combining that equation
with \eqnref{Ctime-general} and the Sandwich \lemref{sandwich}, we obtain
\[
\label{eq:plusminus2}d_W \biggl(\frac{\Ctime_1}{\mathsf{m}(Q)
},\sum
_{i\geq2}\mathsf{Z}_{i} \biggr) = O
\biggl(k^2 \mathsf{err}(Q)+ k^4\mathsf{err}(Q)^2+
\frac {\ln|\bV|}{k} \biggr).
\]
We choose $k=\lceil(\ln|\bV|/\mathsf{err}(Q))^{1/3}\rceil$ to
finish the proof,
at least if this is smaller than $1/5\sqrt{\mathsf{err}(Q)}$. But the
bound in the
theorem is trivial if that is not the case, so we are done.
\end{pf*}

\section{Final remarks}

\begin{itemize}
\item Cooper et al.~\cite{CooperEtAl_IPSOnExpanders} consider many
other processes besides coalescing random walks. It is not hard to
modify our analysis to study those processes over more general graphs,
at least when the initial number of random walks is not too large (this
restriction is also present in~\cite{CooperEtAl_IPSOnExpanders}).
\item Our Theorems~\ref{thm:rareeventsI} and~\ref{thm:rareeventsII} can
be used to study other problems related to hitting times. Alan Prata
and the present author~\cite{Prata_Tese} have used these results to
prove the Gumbel law for the fluctuations of cover times for a large
family of graphs, including all examples where it was previously known.
We have also used extensions of these results to compute the asymptotic
distribution of the $k$ last points to be visited, for any constant
$k$: those are uniformly distributed over the graph, as conjectured by
Aldous and Fill~\cite{AldousFill_RWBook}.
\end{itemize}

\begin{appendix}\label{sec:proof_W}
\section*{Appendix: Proofs of techncal results on $L_1$ Wasserstein
distance}

\subsection{\texorpdfstring{Proof of Sandwich lemma (\protect\lemref{sandwich})}{Proof of Sandwich lemma (Lemma 2.2)}}

Notice that for all $t\in\R$,
\[
\mathbf{P} (Z_-\geq t )\leq\mathbf{P} (Z\geq t )\leq\mathbf{P} (Z_+\geq t ).
\]
By convexity, this implies
\begin{eqnarray*}
\bigl|\mathbf{P} (Z\geq t )-\mathbf{P} (W\geq t )\bigr|&\leq&\bigl|\mathbf{P} (Z_-\geq t )-
\mathbf{P} (W\geq t )\bigr|
\\
&&{}+ \bigl|\mathbf{P} (Z_+\geq t )-\mathbf{P} (W\geq t )\bigr|.
\end{eqnarray*}
Integrate both sides to obtain the result.

\subsection{\texorpdfstring{Proof of conditional lemma (\protect\lemref{conditionalwasserstein})}{Proof of conditional lemma (Lemma 2.3)}}

First notice that the sigma field $\sigma(W_1)$ generated by $W_1$ is
contained in $\sG$. This implies that for all $t\in\R$,
\begin{eqnarray*}
&& \mathbf{E} \bigl[\bigl|\mathbf{P} (W_2
\geq t\mid \sG )- \mathbf{P} (Z_2\geq t )\bigr| \bigr]
\\
&&\qquad= \mathbf{E} \bigl[\mathbf{E} \bigl[\bigl|\mathbf{P} (W_2\geq t\mid \sG
)-\mathbf{P} (Z_2\geq t ) \bigr|| \sigma (W_1) \bigr]\bigr]
\\
&&\qquad\geq \mathbf{E} \bigl[\bigl|\mathbf{P} \bigl(W_2\geq t\mid\sigma
(W_1) \bigr)- \mathbf{P} (Z_2\geq t )\bigr| \bigr].
\end{eqnarray*}
Integrating both sides in $t$ and applying Fubini--Tonelli gives
\[
\mathbf{E} \bigl[d_W\bigl(\operatorname{Law} (W_2\mid
\sG ),\operatorname{Law} (Z_2 )\bigr) \bigr]\geq \mathbf{E}
\bigl[d_W \bigl(\operatorname{Law} \bigl(W_2\mid\sigma
(W_1) \bigr), \operatorname{Law} (Z_2 ) \bigr) \bigr].
\]

Therefore it suffices to prove the theorem in the case $\sG=\sigma
(W_1)$. For simplicity, we will assume that $(Z_1,Z_2,W_1,W_2)$ are all
defined in the same probability space, with $(Z_1,Z_2)$ independent
from $(W_1,W_2)$.
Let $f\dvtx \R\to\R$ be $1$-Lipschitz. We have
\[
\mathbf{E} \bigl[f(W_1+W_2)\mid W_1=w_1
\bigr] = \int f(w_1+w_2) \mathbf{P} (W_2\in
dw_2 \mid W_1=w_1 ).
\]
By the duality version of $d_W$, we have
\begin{eqnarray*}
&&\int f(w_1+w_2) \mathbf{P} (W_2\in
dw_2 \mid W_1=w_1 )
\\
&&\qquad\leq\int f(w_1+z_2) \mathbf{P} (Z_2\in
dz_2 ) + d_W\bigl(\operatorname{Law} (W_2
\mid W_1=w_1 ),\operatorname {Law} (Z_2 )
\bigr).
\end{eqnarray*}
Integrating over $W_1=w_1$ and using the fact that $Z_2$ is independent
from $W_1$, we obtain
\[
\mathbf{E} \bigl[f(W_1+W_2) \bigr]\leq\mathbf{E}
\bigl[f(W_1+Z_2) \bigr] + d_W\bigl(
\operatorname{Law} (W_2\mid W_1 ), \operatorname{Law}
(Z_2 )\bigr).
\]
But we also have
\[
\mathbf{E} \bigl[f(W_1+Z_2)\mid Z_2=z_2
\bigr] = \mathbf{E} \bigl[f(W_1+z_2) \bigr]\leq\mathbf{E}
\bigl[f(Z_1+z_2) \bigr] + d_W(W_1,Z_1),
\]
and the independence of $Z_1,Z_2$ implies
\[
\mathbf{E} \bigl[f(W_1+Z_2) \bigr] \leq\mathbf{E}
\bigl[f(Z_1+Z_2) \bigr] + d_W(W_1,Z_1).
\]
We conclude
\begin{eqnarray*}
\mathbf{E} \bigl[f(W_1+W_2) \bigr] &\leq&\mathbf{E}
\bigl[f(Z_1+Z_2) \bigr] + d_W(W_1,Z_1)
\\
& & {} + d_W\bigl(\operatorname{Law} (W_2\mid
W_1 ), \operatorname{Law} (Z_2 )\bigr).
\end{eqnarray*}
Since $f$ is an arbitrary $1$-Lipschitz function, we are done.
\end{appendix}
\section*{Acknowledgment}

We warmly thank the anonymous referee for pointing out several typos in
a previous versions of this paper.

%
%

%



\printaddresses


\begin{thebibliography}{17}

\bibitem{AldousTalk}
%
\begin{bmisc}[auto:STB|2013/04/24|11:25:54]
\bauthor{\bsnm{Aldous},~\bfnm{David}\binits{D.}}
(\byear{2010}).
\bhowpublished{Mixing times and hitting times.
Available at \url{http://www.stat.berkeley.edu/\textasciitilde aldous/Talks/slides.html}.}
\bptok{imsref}%
\end{bmisc}
%
\endbibitem

\bibitem{AldousFill_RWBook}
%
\begin{bmisc}[auto:STB|2013/04/24|11:25:54]
\bauthor{\bsnm{Aldous},~\bfnm{David}\binits{D.}} \AND
\bauthor{\bsnm{Fill},~\bfnm{James~Allen}\binits{J.~A.}}
(\byear{2001}).
\bhowpublished{Reversible Markov {chains} and random {walks} on
graphs. Available at \url{http://www.stat.berkeley.edu/\textasciitilde aldous/RWG/book.html}.}
\bptok{imsref}%
\end{bmisc}
%
\endbibitem

\bibitem{Aldous_ExpHittingTimes}
%
\begin{barticle}[mr]
\bauthor{\bsnm{Aldous},~\bfnm{David~J.}\binits{D.~J.}}
(\byear{1982}).
\btitle{Markov chains with almost exponential hitting times}.
\bjournal{Stochastic Process. Appl.}
\bvolume{13}
\bpages{305--310}.
\bid{doi={10.1016/0304-4149(82)90016-3}, issn={0304-4149}, mr={0671039}}
\bptok{imsref}%
\end{barticle}
%
\endbibitem

\bibitem{AldousBrown_RareEvents}
%
\begin{bincollection}[mr]
\bauthor{\bsnm{Aldous},~\bfnm{David~J.}\binits{D.~J.}} \AND
\bauthor{\bsnm{Brown},~\bfnm{Mark}\binits{M.}}
(\byear{1992}).
\btitle{Inequalities for rare events in time-reversible {M}arkov
chains. {I}}.
In \bbooktitle{Stochastic Inequalities ({S}eattle, {WA}, 1991)}.
\bseries{Institute of Mathematical Statistics Lecture
Notes---Monograph Series}
\bvolume{22}
\bpages{1--16}.
\bpublisher{IMS}, \blocation{Hayward, CA}.
\bid{doi={10.1214/lnms/1215461937}, mr={1228050}}
\bptok{imsref}%
\end{bincollection}
%
\endbibitem

\bibitem{BenjaminiMossel_RWPercolation}
%
\begin{barticle}[mr]
\bauthor{\bsnm{Benjamini},~\bfnm{Itai}\binits{I.}} \AND
\bauthor{\bsnm{Mossel},~\bfnm{Elchanan}\binits{E.}}
(\byear{2003}).
\btitle{On the mixing time of a simple random walk on the super critical
percolation cluster}.
\bjournal{Probab. Theory Related Fields}
\bvolume{125}
\bpages{408--420}.
\bid{doi={10.1007/s00440-002-0246-y}, issn={0178-8051}, mr={1967022}}
\bptok{imsref}%
\end{barticle}
%
\endbibitem

\bibitem{CooperEtAl_IPSOnExpanders}
%
\begin{barticle}[mr]
\bauthor{\bsnm{Cooper},~\bfnm{Colin}\binits{C.}},
\bauthor{\bsnm{Frieze},~\bfnm{Alan}\binits{A.}} \AND
\bauthor{\bsnm{Radzik},~\bfnm{Tomasz}\binits{T.}}
(\byear{2009}).
\btitle{Multiple random walks in random regular graphs}.
\bjournal{SIAM J. Discrete Math.}
\bvolume{23}
\bpages{1738--1761}.
\bid{doi={10.1137/080729542}, issn={0895-4801}, mr={2570201}}
\bptok{imsref}%
\end{barticle}
%
\endbibitem

\bibitem{Cox_Coalescing}
%
\begin{barticle}[mr]
\bauthor{\bsnm{Cox},~\bfnm{J.~T.}\binits{J.~T.}}
(\byear{1989}).
\btitle{Coalescing random walks and voter model consensus times on the
torus in~{$\mathbb{Z}^d$}}.
\bjournal{Ann. Probab.}
\bvolume{17}
\bpages{1333--1366}.
\bid{issn={0091-1798}, mr={1048930}}
\bptok{imsref}%
\end{barticle}
%
\endbibitem

\bibitem{Durrett_RGDynamics}
%
\begin{bbook}[mr]
\bauthor{\bsnm{Durrett},~\bfnm{Rick}\binits{R.}}
(\byear{2007}).
\btitle{Random Graph Dynamics}.
\bpublisher{Cambridge Univ. Press}, \blocation{Cambridge}.
\bid{mr={2271734}}
\bptnote{check year}%
\bptok{imsref}%
\end{bbook}
%
\endbibitem

\bibitem{Durrett_PNAS}
%
\begin{barticle}[auto:STB|2013/04/24|11:25:54]
\bauthor{\bsnm{Durrett},~\bfnm{Rick}\binits{R.}}
(\byear{2010}).
\btitle{Some features of the spread of epidemics and information on a random
graph}.
\bjournal{Proc. Natl. Acad. Sci. USA}
\bvolume{107}
\bpages{4491--4498}.
\bptok{imsref}%
\end{barticle}
%
\endbibitem

\bibitem{FountoulakisReed_RWGiantComponent}
%
\begin{barticle}[mr]
\bauthor{\bsnm{Fountoulakis},~\bfnm{N.}\binits{N.}} \AND
\bauthor{\bsnm{Reed},~\bfnm{B.~A.}\binits{B.~A.}}
(\byear{2008}).
\btitle{The evolution of the mixing rate of a simple random walk on
the giant
component of a random graph}.
\bjournal{Random Structures Algorithms}
\bvolume{33}
\bpages{68--86}.
\bid{doi={10.1002/rsa.20210}, issn={1042-9832}, mr={2428978}}
\bptok{imsref}%
\end{barticle}
%
\endbibitem

\bibitem{LeskovecEtAl_Community}
%
\begin{barticle}[mr]
\bauthor{\bsnm{Leskovec},~\bfnm{Jure}\binits{J.}},
\bauthor{\bsnm{Lang},~\bfnm{Kevin~J.}\binits{K.~J.}},
\bauthor{\bsnm{Dasgupta},~\bfnm{Anirban}\binits{A.}} \AND
\bauthor{\bsnm{Mahoney},~\bfnm{Michael~W.}\binits{M.~W.}}
(\byear{2009}).
\btitle{Community structure in large networks: {N}atural cluster sizes
and the
absence of large well-defined clusters}.
\bjournal{Internet Math.}
\bvolume{6}
\bpages{29--123}.
\bid{issn={1542-7951}, mr={2736090}}
\bptok{imsref}%
\end{barticle}
%
\endbibitem

\bibitem{LevinPeresWilmer_MCBook}
%
\begin{bbook}[mr]
\bauthor{\bsnm{Levin},~\bfnm{David~A.}\binits{D.~A.}},
\bauthor{\bsnm{Peres},~\bfnm{Yuval}\binits{Y.}} \AND
\bauthor{\bsnm{Wilmer},~\bfnm{Elizabeth~L.}\binits{E.~L.}}
(\byear{2009}).
\btitle{Markov Chains and Mixing Times}.
\bpublisher{Amer. Math. Soc.}, \blocation{Providence, RI}.
\bid{mr={2466937}}
\bptnote{check year}%
\bptok{imsref}%
\end{bbook}
%
\endbibitem

\bibitem{Liggett_IPSBook}
%
\begin{bbook}[mr]
\bauthor{\bsnm{Liggett},~\bfnm{Thomas~M.}\binits{T.~M.}}
(\byear{1985}).
\btitle{Interacting Particle Systems}.
\bseries{Grundlehren der Mathematischen Wissenschaften}
\bvolume{276}.
\bpublisher{Springer}, \blocation{New York}.
\bid{mr={0776231}}
\bptok{imsref}%
\end{bbook}
%
\endbibitem

\bibitem{Oliveira_TAMS}
%
\begin{barticle}[mr]
\bauthor{\bsnm{Oliveira},~\bfnm{Roberto~Imbuzeiro}\binits{R.~I.}}
(\byear{2012}).
\btitle{On the coalescence time of reversible random walks}.
\bjournal{Trans. Amer. Math. Soc.}
\bvolume{364}
\bpages{2109--2128}.
\bid{doi={10.1090/S0002-9947-2011-05523-6}, issn={0002-9947}, mr={2869200}}
\bptnote{check year}%
\bptok{imsref}%
\end{barticle}
%
\endbibitem

\bibitem{Pete_ConnectivityMixingPercolation}
%
\begin{barticle}[mr]
\bauthor{\bsnm{Pete},~\bfnm{G{\'a}bor}\binits{G.}}
(\byear{2008}).
\btitle{A note on percolation on $\mathbb{Z}^{d}$: Isoperimetric
profile via
exponential cluster repulsion}.
\bjournal{Electron. Commun. Probab.}
\bvolume{13}
\bpages{377--392}.
\bid{doi={10.1214/ECP.v13-1390}, issn={1083-589X}, mr={2415145}}
\bptok{imsref}%
\end{barticle}
%
\endbibitem

\bibitem{Prata_Tese}
%
\begin{bmisc}[auto:STB|2013/04/24|11:25:54]
\bauthor{\bsnm{Prata},~\bfnm{Alan}\binits{A.}}
(\byear{2012}).
\bhowpublished{Stochastic processes over finite graphs. {Ph.D.}
thesis in
Mathematics. IMPA, Rio de Janeiro, Brazil}.
\bptok{imsref}%
\end{bmisc}
%
\endbibitem

\end{thebibliography}
\end{document}